\documentclass[a4paper,12pt]{amsart}
\usepackage{amsthm,amsmath, amssymb, amscd,mathtools, braket}
\usepackage[top=30truemm, bottom=30truemm, left=25truemm, right=25truemm]{geometry}
\usepackage{color}
\usepackage{stmaryrd}
\usepackage{caption}
\usepackage{graphicx}

\newcommand{\N}{\mathbb{N}}
\newcommand{\Z}{\mathbb{Z}}
\newcommand{\floor}[1]{\left\lfloor #1 \right\rfloor}
\newcommand{\ceil}[1]{\left\lceil #1 \right\rceil}

\newcommand{\pas}[1]{\left(#1\right)}
\newcommand{\Pas}[1]{\left\{#1\right\}}

\newcommand{\bbracket}[1]{\left\llbracket #1 \right\rrbracket }

\newcommand{\PS}{\mathrm{PS}}

\newcommand{\upperT}{\overline{T}}
\newcommand{\upperS}{\overline{S}}

\numberwithin{equation}{section}
\newcounter{count}

\newcommand{\num}{\stepcounter{count}\the\value{count}}

\renewcommand{\epsilon}{\varepsilon}

\renewcommand{\liminf}{\varliminf}

\newtheorem{theorem}{Theorem}[section]
\newtheorem{lemma}[theorem]{Lemma}
\newtheorem{corollary}[theorem]{Corollary}

\newtheorem{conjecture}[theorem]{Conjecture}
\newtheorem{question}[theorem]{Question}

\theoremstyle{definition}
\newtheorem{remark}[theorem]{Remark}
\newtheorem{notation}[theorem]{Notation}

\begin{document}

\title[Diophantine equations on analogs of squares]{A system of certain linear Diophantine\\ equations on analogs of squares}

\author[Y. Kanado]{Yuya Kanado }
\address{Yuya Kanado\\
Graduate School of Mathematics\\ Nagoya University\\ Furo-cho\\ Chikusa-ku\\ Nagoya\\ 464-8602\\ Japan}
\email{m21017a@math.nagoya-u.ac.jp}

\author[K. Saito]{Kota Saito}
\address{Kota Saito\\Faculty of Pure and Applied Sciences\\ University of Tsukuba\\ 1-1-1 Tennodai\\ Tsukuba\\ Ibaraki\\ 305-8577\\ Japan}
\email{saito.kota.gn@u.tsukuba.ac.jp}

\thanks{}


\begin{abstract}
This study investigates the existence of tuples $(k, \ell, m)$ of integers such that all of $k$, $\ell$, $m$, $k+\ell$, $\ell+m$, $m+k$, $k+\ell+m$ belong to $S(\alpha)$, where $S(\alpha)$ is the set of all integers of the form $\lfloor \alpha n^2 \rfloor$ for $n\geq \alpha^{-1/2}$ and $\lfloor x\rfloor$ denotes the integer part of $x$.   We show that $T(\alpha)$, the set of all such tuples, is infinite for all $\alpha\in (0,1)\cap \mathbb{Q}$ and for almost all $\alpha\in (0,1)$ in the sense of the Lebesgue measure. Furthermore, we show that if there exists $\alpha>0$ such that $T(\alpha)$ is finite, then there is no perfect Euler brick. We also examine the set of all integers of the form $\lceil \alpha n^2 \rceil$ for $n\in \mathbb{N}$.
\end{abstract}


\maketitle

\section{Introduction}
Let $\mathbb{N}$ be the set of all positive integers, and $S$ be the set of all perfect squares. A rectangular cuboid is called an \textit{Euler brick} if the edges and face diagonals have integral lengths. Further, an Euler brick is called a \textit{perfect Euler brick} if the space diagonal also has integral length. It is known that there are infinitely many Euler bricks. For example, let $(a,b,c)$ be a tuple of the side lengths of an Euler brick. Then we have
\[
(a,b,c)= (240, 117, 44),\ (275, 252, 240),\ (693, 480, 140),\ (720, 132, 85),\ \ldots.
\]
We can see the list of the side lengths of Euler bricks in OEIS A031173, A031174, and A031175. However, the existence (or non-existence) of a perfect Euler brick is unknown. It is a long-standing problem. By the Pythagorean theorem, a perfect Euler brick exists if and only if there exists $(k,\ell,m)\in \mathbb{N}^3$ such that all of
\begin{equation}\label{terms-klm}
    k,\quad \ell,\quad m,\quad k+\ell,\quad \ell+m,\quad m+k,\quad k+\ell+m
\end{equation}
belong to $S$. Instead of squares, Glasscock investigated the Piatetski-Shapiro sequences \cite{Glasscock17}. Let $\lfloor x\rfloor$ denote the integer part of $x$ for all $x\in \mathbb{R}$. For every non-integral $\alpha>1$, the sequence $(\lfloor n^\alpha \rfloor)_{n\in \mathbb{N}} $ is called the \textit{Piatestki-Shapiro sequence with exponent} $\alpha$, and we denote the set of all terms of this sequence by $\PS(\alpha)$. For a given set $X\subseteq \mathbb{N}$, we define $T(X)$ as the set of all tuples $(k,\ell ,m )\in \mathbb{N}^3$ with $k\leq \ell \leq m$ such that all of \eqref{terms-klm}  belong to $X$.

  Interestingly, Glasscock found that $\#T(\PS(\alpha))=\infty$ for almost all $\alpha \in (1,2)$ \cite[Corollary~1]{Glasscock17}. Note that $\PS(2)=S$. The second author improved on this finding, showing that Glasscock's result remains true even if we replace ``for almost all'' with ``for all''; that is, $\#T(\PS(\alpha))=\infty$ for all $\alpha \in (1,2)$ \cite[Corollary~1.2]{Saito22}. In view of these previous results, it is natural to consider the following question.

\begin{question}\label{Question-1}
Can we find a sequence $(s_n)_{n\in \mathbb{N}}$ of integers satisfying that each term $s_n$ is much closer to $n^2$ than $\lfloor n^\alpha \rfloor$ and $\# T (\{s_n\colon n\in \mathbb{N}\})=\infty$?
\end{question}
As an answer to Question~\ref{Question-1}, we propose the following set:
\begin{align*}
S(\alpha) :=\{\lfloor \alpha n^2 \rfloor \colon n\geq \alpha^{-1/2},\ n\in \mathbb{N}  \}. \
\end{align*}
Let $T (\alpha):= T(S (\alpha))$, and define
\[
T_{\leq x} (\alpha):=\{(\lfloor \alpha n_1^2 \rfloor,\lfloor \alpha n_2^2 \rfloor, \lfloor \alpha n_3^2 \rfloor) \in T(\alpha) \colon  n_1,n_2,n_3\in [\alpha^{-1/2},  x]\cap \mathbb{N}  \} 
\]
for all $\alpha>0$ and $x>1$. Note that $S(1)=S$. A linear relation on $S(\alpha)$ has  already been studied in \cite[Appendix~A]{SYoshida}, written by the second author and Yoshida. According to their result \cite[Theorem~15]{SYoshida}, for every integer $k\geq 3$, there exists a sufficiently small $\alpha_k>0$ such that  $S(\alpha)$ contains infinitely many arithmetic progressions of length $k$ for all $0<\alpha \leq \alpha_k$.

 The goal of this article is to develop the theorems shown below. Before stating the theorems, however, a few preliminaries are necessary. First, let $d(\alpha)$ denote the positive denominator of the irreducible fraction of $\alpha$ for all $\alpha\in\mathbb{Q}$. Furthermore, for all $x\in \mathbb{R}$, let $\lceil x\rceil =-\lfloor -x\rfloor$, which means the least integer greater than or equal to $x$. 
\begin{theorem}\label{Theorem-lowerT-Q}
For all $\alpha \in (0,1) \cap \mathbb{Q}$ and $x>1$,  we have 
\begin{equation}\label{Inequality-T(alpha)-Q}
\# T_{\leq x}(\alpha)\geq \floor{\frac{\log(16x) }{4\ceil{\sqrt{2} d(\alpha)} \log (1+\sqrt{2})} } . 
\end{equation}
In particular, $T(\alpha)$ is infinite for all $\alpha\in (0,1)\cap \mathbb{Q}$.
\end{theorem}
 As a consequence of this theorem, we obtain the following corollary.  
\begin{corollary}\label{Corollary-lowerT-positive}
There exists an open and dense set $U\subseteq (0,1)$ such that $\# T(\alpha)>0$ for all $\alpha\in U$.
\end{corollary}
We will prove Theorem~\ref{Theorem-lowerT-Q} and Corollary~\ref{Corollary-lowerT-positive} in Section~\ref{Section-ProofTheoremCorollary}. Furthermore, we show the following metric result.
\begin{theorem}\label{Theorem-lowerT-almostall}
 Let $0<s<t<1$ be real numbers. For almost all $\alpha \in (s,t)$ and for every $\epsilon>0$, there exists $x_0=x_0(\alpha,s,t,\epsilon)>1$ such that for every $x\geq x_0$, we have
 \[
 \# T_{\leq x}(\alpha)  \geq (1-t)\min \left( 1-t,  2s\right) \frac{\log (16x)}{ 16 \log (1+\sqrt{2})+\epsilon}.
 \]
 In particular, $T(\alpha)$ is infinite for almost all $\alpha\in (0,1)$.
\end{theorem}

In addition, we obtain quantitative upper bounds for $\# T_{\leq x}(\alpha) $.
\begin{theorem}\label{Theorem-Upperbounds}
Let $\alpha>0$. There exists $C=C(\alpha)>0$ such that for all $x\geq 2$, we have $\#T_{\leq x}(\alpha)\leq C x(\log x)^{15}$.
\end{theorem}

The main contribution of our results is to disclose the infinitude of $T(\alpha)$ for all $\alpha\in (0,1)\cap \mathbb{Q}$ and for almost all $\alpha\in (0,1)$ even though the existence of an element in $T(1)$ is unknown. Furthermore, we obtain quantitative upper and lower bounds for the cardinality of $T_{\leq x}(\alpha)$. We will show Theorem~\ref{Theorem-lowerT-almostall} and Theorem~\ref{Theorem-Upperbounds} in Section~\ref{Section-UniformDistribution} and Section~\ref{Section-Upperbounds}, respectively.  Note that we cannot prove the existence of $\alpha $ satisfying the condition that $T(\alpha)$ is finite. Interestingly, the existence of such $\alpha$ implies the non-existence of a perfect Euler brick as follows.  

\begin{theorem}\label{Theorem-infinitePEBs}
Assume that there is $\alpha >0$ such that 
\begin{equation}\label{Equation-limitinfT}
\liminf_{x\to \infty} \#T_{\leq x}(\alpha)/x=0.
\end{equation}
Then there is no perfect Euler brick. In particular, if there is $\alpha >0$ such that $T(\alpha)$ is finite, then there is no perfect Euler brick.
\end{theorem}
From Theorem~\ref{Theorem-infinitePEBs}, we can discover the importance of finding upper bounds for $\#T_{\leq x}(\alpha)$. Note that Theorem~\ref{Theorem-Upperbounds} implies $\# T_{\leq x}(\alpha)/x \leq  C(\log x)^{15}$. This does not reach \eqref{Equation-limitinfT}. We will show Theorem~\ref{Theorem-infinitePEBs} in Section~\ref{Section-Homogeneous}. More generally, we will examine a system of homogeneous Diophantine equations. For example, let $f_j(x_1,\ldots ,x_r)=\sum_{\nu=1}^r a^{(j)}_\nu x_\nu^{d_j}$ for every $1\leq j\leq K$, where all the coefficients $a^{(j)}_\nu$ are integers and the degrees $d_j$ are positive integers. Then for all $\alpha \in \mathbb{R}$ and $x_1,\ldots ,x_r \in \mathbb{R}$, we define 
\[
\floor{\alpha f_j}(x_1,\ldots ,x_r)= \sum_{\nu=1}^r a^{(j)}_\nu \floor{ \alpha x_\nu^{d_j}}.
\]
In Section~\ref{Section-Homogeneous}, we will prove that if there is $\alpha \in \mathbb{R}$ such that 
\begin{align*}
\liminf_{N\to \infty} \frac{1}{N} \#\{ &(y_1,\ldots ,y_r)\in ([1,N]\cap \mathbb{N})^r\colon  \\
&\floor {\alpha f_j}(y_1,\ldots ,y_r)=0 \text{ for all $1\leq j\leq K$}\}=0,
\end{align*}
then there is no $(x_1,\ldots, x_r)\in \mathbb{N}^r$ such that $f_j(x_1,\ldots ,x_r)=0$ for every $1\leq j\leq K$.

In the appendix, we will investigate the set 
\[
\upperS(\alpha) :=\{\lceil \alpha n^2 \rceil \colon n\in \mathbb{N}\}
\]
for every $\alpha\in (0,1)$. In the case of the floor function, for every fixed $n\in \mathbb{N}$, we have $\lim_{\alpha\to 1^-} \lfloor \alpha n^2 \rfloor=n^2-1$. Thus, there is an absolute gap between each of the terms in $S(\alpha)$ and $S$ even if $\alpha$ is very close to $1$. On the other hand, by considering the ceiling function, we have $\lim_{\alpha\to 1^-} \ceil{ \alpha n^2} =n^2$. Therefore, there is no gap between each of the terms in $\upperS(\alpha)$ and $S$ if $\alpha$ is sufficiently near $1$. From this observation, it seems to be more natural to study  $\upperS(\alpha)$ than $S(\alpha)$. However, the problems involving $\upperS(\alpha)$ are much more difficult. For example, we will prove that 
$\#T(\upperS (\alpha)) =\infty$ for almost all $\alpha\in (0,2/3)$. In this article, we are unable to prove the infinitude of $T(\upperS (\alpha))$ in the case $\alpha\in [2/3,1)$ due to technical problems.

\begin{notation}
Let $\mathbb{Z}_{\geq 0}$ be the set of all non-negative integers. For $x\in \mathbb{R}$, let $\{x\}$ denote the fractional part of $x$. For all intervals $I\subset \mathbb{R}$, let  $I_\mathbb{Z}=I\cap \mathbb{Z}$. Further, we write $[N]=[1,N]_\mathbb{Z}$ for every $N\in \mathbb{N}$. Let $P(x)$ be a statement depending on $x\in I$, where $I\subseteq \mathbb{R}$ is an interval. Then we say that $P(x)$ is true for almost all $x\in I$ if $\{x\in I \colon \text{$P(x)$ is not true} \}$ is a null set in the sense of the $1$-dimensional Lebesgue measure.   
 \end{notation}

\section{Ingredients of the proofs} \label{Section-Outline}
Let us firstly discuss the sketch of proofs of Theorem~\ref{Theorem-lowerT-Q} and Theorem~\ref{Theorem-lowerT-almostall}. We make heuristic calculations to construct elements in $T(\alpha)$. From previous work \cite{Glasscock17, Saito22}, the following two items are important to prove the infinitude of $T(\PS(\alpha))$:
\begin{enumerate}
    \item[(i)]\label{Ingredient-DA} the rational approximations of $2^{1/\alpha}$;
    \item[(ii)]\label{Ingredient-AP} the existence of arithmetic progressions of length $3$ of $\PS(\alpha)$ with gap difference $d$ for sufficiently large $d$. 
\end{enumerate}
Indeed, let us now apply (i). Fix any suitable $\alpha\in (1,2)$. From (i), we find infinitely many pairs of integers $(p_n,q_n)$ such that $p_n/q_n$ is close to $2^{1/\alpha}$. Actually, by Dirichlet's approximation theorem, there exist infinitely many pairs $(p_n,q_n)$ of integers with $0<q_1<q_2<\cdots$ such that $|p_n/q_n - 2^{1/\alpha} | <q_n^{-2}$. Roughly speaking, $p_n^\alpha $ is close to $2q_n^\alpha$. Therefore, we deduce that
\begin{equation}\label{Equation-PS1}
\lfloor p_n^\alpha \rfloor = 2 \lfloor q_n^\alpha \rfloor
\end{equation}
by controlling the fractional parts of $p_n^\alpha $ and $q_n^\alpha$, which is the most difficult part of this proof. We omit the details. 

 Let us next apply (ii). By the result of Frantzikinakis and Wierdl \cite[Proposition~5.1]{FrantzikinakisWierdl}, for every integer $k\geq 3$ and $\alpha\in (1,2)$, there exists a sufficiently large $d_0=d_0(\alpha,k)>0$ such that for every $d\geq d_0$, we can find an arithmetic progression of $\PS(\alpha)$ with length $k$ and gap difference $d$. By this result, for sufficiently large $n\in \mathbb{N}$, there exists an arithmetic progression of $\PS(\alpha)$ with length $3$ and gap difference $\floor{q_n^\alpha}$. Therefore, there exists $r_n\in \mathbb{N}$ such that all of
 \begin{equation}\label{Terms-PSAPs}
 \floor{ r_n^\alpha},\quad  \floor{ r_n^\alpha}+ \floor{q_n^\alpha},\quad \floor{ r_n^\alpha}+ 2\floor{q_n^\alpha}
 \end{equation}
belong to $\PS(\alpha)$. Set $k=\ell = \floor{q_n^\alpha}$ and $m=\floor{ r_n^\alpha}$. Then all of 
\[
k,\quad \ell, \quad m,\quad k+\ell,\quad \ell +m, \quad m+k,\quad k+\ell+m  
\]
belong to $\PS(\alpha)$. 
Indeed, by the choice of $k$, $\ell$, $m$, it is clear that $k,\ell,m\in \PS(\alpha)$. In addition, by \eqref{Equation-PS1}, $k+\ell =2\lfloor q_n^\alpha \rfloor =\lfloor p_n^\alpha \rfloor\in \PS(\alpha)$. Further, since all terms in \eqref{Terms-PSAPs} belong to $\PS(\alpha)$, $\ell+m\: (=k+\ell)$ and $k+\ell+m$ also belong to $\PS(\alpha)$. Therefore, $(k,\ell ,m)$ or $(m,k,\ell )$ is in $T(\PS(\alpha))$. Hence, we deduce the infinitude of $T(\PS(\alpha))$ by substituting $n$ for every sufficiently large integer.  

Based on this discussion, the main ingredients in our case are the following:
\begin{enumerate}
    \item[(i')] the rational approximations of $\sqrt{2}$;
    \item[(ii')] the existence of arithmetic progressions of length $3$ of $S(\alpha)$ with gap difference $d$ for sufficiently large $d$. 
\end{enumerate}
Let us discuss (i'). It is well known that the continued fraction expansion of $\sqrt{2}$ is 
\[
1+ \frac{1 }{2+ \cfrac{1}{2+\cfrac{1}{\cdots } }  }=[1; 2,2,\cdots].
\]
For instance, see \cite[Example~1.3.4]{DajaniKraaikamp}. Let $G_n/P_n$ denote the irreducible fraction of $[\overbrace{1; 2,\cdots,2}^{n}]$. By the theory of continued fractions  \cite[(1.14)]{DajaniKraaikamp}, we obtain 
\[
|G_n/P_n -\sqrt{2} | <P_n^{-2}.
\]
Therefore, $G_n^2$ is close to $2P_n^2$. Actually, $G_n^2-2P_n^2=(-1)^n$ for every $n\in \mathbb{N}$, which we will show in Section~\ref{Section-Pell}. Furthermore, for every $n\in \mathbb{N}$, we observe that
\[
\frac{G_{n+1}}{ P_{n+1}} = 1+ \frac{1}{1+G_n/P_n }= \frac{G_n+2P_n}{G_n+P_n}. 
\]
Since $\gcd(G_n+2P_n, G_n+P_n)=1$, we have $G_{n+1}=G_n +2P_n$ and $P_{n+1}= G_n+P_n$ for every $n\in \mathbb{N}$. Hence, $G_n=P_{n+1}-P_n$ and 
\begin{equation}
P_{n+1}=2P_n+P_{n-1} 
\end{equation}
for every $n\in \mathbb{N}$, where we define $P_0=0$ and $P_1=1$. The sequence $(P_n)_{n\geq 0}$ is called the \textit{Pell sequence}. Instead of $p_n^\alpha$ and $q_n^\alpha$ as in previous studies, we consider $\alpha G_n^2$ and $\alpha P_n^2$. Fix any suitable $\alpha\in (0,1)$. By controlling the fractional parts of $\alpha G_n^2$ and $\alpha P_n^2$, we obtain 
\begin{equation}
\floor{\alpha G_n^2} = 2\floor{\alpha P_n^2}.  
\end{equation}
This corresponds to \eqref{Equation-PS1}. 

By discussing (ii'), let us next show that there exists $ H_n\in \mathbb{N}$ with $H_n \geq P_n$ such that all of 
\begin{equation}\label{Relation-HF}
\floor{\alpha H_n^2},\quad  \floor{\alpha H_n^2}+ \floor{\alpha P_n^2},\quad \floor{\alpha H_n^2}+ 2\floor{\alpha P_n^2}
\end{equation}
belong to $S(\alpha)$. If such $H_n$ exists, then we obtain 
\begin{equation}\label{Relation-kkm}
(\floor{\alpha P_n^2}, \floor{\alpha P_n^2}, \floor{\alpha H_n^2}  )\in T(\alpha).
\end{equation}
 
However, we cannot apply the result given by Frantzikinakis and Wierdl to $S (\alpha)$. As an alternative, we consider the equation 
\[
\alpha(P_n^2/2+j)^2 = \alpha (P_n^2/2)^2 +  j \alpha P_n^2 + \alpha j^2= \floor {\alpha (P_n^2/2)^2} +    j \floor{\alpha P_n^2}+ \delta (\alpha, n, j)  
\]
for every $j\in \{-1,0,1 \}$, where we define $\delta (\alpha, n, j)  =  \Pas{ \alpha (P_n^2/2)^2} + j \Pas{\alpha P_n^2}+ \alpha j^2 $. Remark that we choose $n\in \mathbb{N}$ such that $P_n$ is even since $P_n^2/2$ should be an integer. By investigating the joint distribution of $\Pas{ \alpha (P_n^2/2)^2}$ and $\Pas{\alpha P_n^2}$, we find that the set 
\[
A=\{n\in \mathbb{N}\colon 0\leq \delta (\alpha, n, j)<1 \text{ for all $j\in \{-1,0,1\}$ and $P_n$ is even}  \}
\]
has positive density. Further, for all $n\in A$ and $j\in \{-1,0,1\}$
\[
\floor{ \alpha(P_n^2/2+j)^2} = \floor {\alpha (P_n^2/2)^2} +    j \floor{\alpha P_n^2}.
\]
Thus, by setting $H_n=P_n^2/2-1$, all of \eqref{Relation-HF} are in $S(\alpha)$ for all $n\in A$. Therefore, for every $x>1$, by counting the number of $n\in A$ such that $H_n\leq x$, we obtain lower bounds for $\#T_{\leq x} (\alpha)$; in particular, we conclude the infinitude of $T(\alpha)$.

Note that it is not easy to control the fractional parts. In the proof of Theorem~\ref{Theorem-lowerT-Q}, we will use the properties of the Pell sequence. In Theorem~\ref{Theorem-lowerT-almostall}, we will apply the theory of the uniform distribution.   \\

The proof of Theorem~\ref{Theorem-infinitePEBs} is based on a different idea. Assume that $\# T(1)>0$. Then by this assumption, there exist $k, \ell, m, a, b, c, d\in \mathbb{N}$ such that 
\begin{gather*}
    k^2+\ell^2 =a^2,\quad \ell^2+m^2=b^2,\quad m^2+k^2=c^2,\quad k^2+\ell^2+m^2=d^2.
\end{gather*}
Fix an arbitrary $\alpha>0$. For every $n\in \mathbb{N}$, we set 
\[
\eta(\alpha,n) = \max_{x\in \{k, \ell, m, a, b, c, d\}} \Pas{\alpha (nx)^2}   .
\]
Then we observe that for all $n\in \mathbb{N}$,
\begin{gather*}
   \left| \floor {\alpha (nk)^2}+\floor{\alpha (n\ell)^2} - \floor{\alpha (na)^2} \right|\leq 3 \eta(\alpha, n),\\
   \left| \floor {\alpha (n\ell)^2}+\floor{\alpha (nm)^2} - \floor{\alpha (nb)^2} \right|\leq 3 \eta(\alpha, n),\\
   \left| \floor {\alpha (nm)^2}+\floor{\alpha (nk)^2} - \floor{\alpha (nc)^2} \right|\leq 3 \eta(\alpha, n),\\
   \left| \floor {\alpha (nk)^2}+\floor{\alpha (n\ell)^2}+\floor{\alpha (nm)^2} - \floor{\alpha (nd)^2} \right|\leq 4 \eta(\alpha, n).
\end{gather*}
Therefore, by the theory of uniform distribution, we can show that the density of the set of $n\in \mathbb{N}$ such that $0\leq \eta (\alpha, n)<1/4$ is positive;  such an $n$ satisfies 
\[
(\floor {\alpha (nk)^2}, \floor{\alpha (n\ell)^2}, \floor {\alpha (nm)^2}) \in T(\alpha).
\]
Hence, we conclude the positiveness of the limit inferior of $\#T_{\leq x}(\alpha)/x$. This idea comes from the work of Matsusaka and the second author \cite{MatsusakaSaito}. They showed that for any fixed $a,b,c\in \mathbb{N}$, there are uncountably many $\alpha>2$ such that the linear equation $ax+by=cz$ has infinitely many solutions $(x,y,z)\in \PS(\alpha)^3$. \\

The remainder of the article is organized as follows. In Section \ref{Section-Pell}, we describe the useful properties of the Pell sequence. In Section~\ref{Section-ProofTheoremCorollary},  we give proofs of Theorem~\ref{Theorem-lowerT-Q} and Corollary~\ref{Corollary-lowerT-positive}. In Section~\ref{Section-UniformDistribution}, we discuss the theory of uniform distribution of sequences and develop Theorem~\ref{Theorem-lowerT-almostall}. In Section~\ref{Section-Upperbounds}, we discuss upper bounds for $\#T_{\leq x}(\alpha)$.  In Section~\ref{Section-Homogeneous}, we provide a proof of Theorem~\ref{Theorem-infinitePEBs}. Finally, in Section~\ref{Section-FurtherDiscussion}, we present numerical results and offer relevant conjectures.

\section{Preparations for the Pell sequence}\label{Section-Pell}

In this section, we describe the basic properties of the Pell sequence. Since the definition of the Pell sequence is similar to the Fibonacci sequence, they share many features. We refer to a book written by Koshy \cite{Koshy} for details on the Fibonacci sequence. 

\begin{lemma}\label{Lemma-Pell} The Pell sequence has the following properties:
\begin{enumerate}
\item\label{Property-Pell0} for every $m\in \mathbb{Z}_{\geq 0}$ and $n\in  \mathbb{N}$, $P_{m+n}=P_{m+1}P_n +P_m P_{n-1}$;
\item\label{Property-Pell1} for every $a,b\in \mathbb{N}$, if $a\mid b$, then $P_a\mid P_b$;
\item\label{Property-Pell2} for every $n\in \mathbb{N}$, $P_{n+1}P_{n-1} -P_n^2= (-1)^n$; 
\item\label{Property-Pell3} for every $n\in \mathbb{Z}_{\geq 0}$, $P_n=\frac{1}{2\sqrt{2}} (\phi_P^n - (-\phi_P)^{-n} )$, where $\phi_P= 1+\sqrt{2}$.
\end{enumerate}
\end{lemma}

\begin{proof}
Let us firstly show \eqref{Property-Pell0}. It follows that $P_{m+1}P_1 +P_mP_0 = P_{m+1}$ for all $m\in \mathbb{Z}_{\geq 0}$. Therefore, \eqref{Property-Pell0} is true for $n=1$. Assume that \eqref{Property-Pell0} is true for all $m\in \mathbb{N}$ and for some $n=k$, where $k\geq 1$. Then by the inductive hypothesis, 
\begin{align*}
P_{m+k+1}&= P_{(m+1)+k}=P_{m+2}P_k +P_{m+1}P_{k-1}= 2P_{m+1}P_k + P_mP_k+ P_{m+1}P_{k-1}\\
&= P_{m+1} (2P_{k} +P_{k-1} ) +P_mP_k= P_{m+1} P_{k+1}   + P_m P_{k} . 
\end{align*}
Hence, we conclude \eqref{Property-Pell0}. 

Let us secondly show \eqref{Property-Pell1}. From $a\mid b$, there exists an integer $c\geq 1$ such that $b=ac$. Then by \eqref{Property-Pell0}, we obtain
\begin{align*}
   P_b&= P_{ac}=P_{a(c-1)+a}= P_{a(c-1)+1} P_a + P_{a(c-1)} P_{a-1}  \\
   &\equiv  P_{a(c-1)} P_{a-1} \equiv \cdots \equiv P_a P_{a-1}^{c-1}\equiv 0 \mod P_a,
\end{align*}
which means that $P_a \mid P_b$. 

Let us next show \eqref{Property-Pell2}. It follows that $P_{2}P_{0} -P_1^2= -P_1^2=-1$. Thus, \eqref{Property-Pell2} is true for $n=1$. Assume that  \eqref{Property-Pell2} is true for $n=k$, where $k\geq 1$. Then 
\begin{align*}
P_{k+2} P_k &=  (2P_{k+1}+P_k) P_k= 2P_{k+1} P_k + P_k^2   \\
&= 2P_{k+1} P_k + P_{k+1}P_{k-1} +(-1)^{k+1} \\
&= P_{k+1} (2 P_k+P_{k-1}) +(-1)^{k+1}  = P_{k+1}^2 +(-1)^{k+1}.  
\end{align*}
By induction, we conclude \eqref{Property-Pell2}. 

Let us lastly show \eqref{Property-Pell3}. Since $x^2-2x-1=0$ if and only if $x\in \{\phi_P , -\phi_P^{-1}\}$, there exist $a,b\in \mathbb{R}$ such that $P_n= a\phi_P^n +b(-\phi_P)^{-n}$ for all $n\geq 0$. By substituting $n=0$ and $n=1$, we obtain $0=a+b$ and $1=a\phi_P+b(-\phi_P)^{-1}$. This implies that $a=-b=1/(2\sqrt{2})$. 
\end{proof}

The opposite direction of the implication \eqref{Property-Pell1} is also true; that is, if $P_a\mid P_b$, then $a\mid b$. However, we do not use this implication in this article.  Note that the Fibonacci sequence also satisfies \eqref{Property-Pell0} to \eqref{Property-Pell2}. These facts can be seen in \cite[Corollary~32.2, Theorem~16.1, Theorem~5.3]{Koshy}.

\begin{lemma}\label{Lemma-PellEquation}
For every $n\in \mathbb{Z}_{\geq 0}$, we have $(P_{n+1} -P_n)^2- 2 P_n^2 = (-1)^n$.
\end{lemma}

\begin{proof}
Fix an arbitrary $n\in \mathbb{Z}_{\geq 0}$. Then we obtain
\begin{align*}
    (P_{n+1} -P_n)^2- 2 P_n^2&= P_{n+1}^2 -2 P_{n+1} P_n +P_n^2 -2P_n^2 = P_{n+1}^2 -2 P_{n+1} P_n -P_n^2\\
    &= P_{n+1}^2 -P_n (2 P_{n+1}  +P_n) = P_{n+1}^2-P_n P_{n+2} =  (-1)^{n},   
\end{align*}
where we apply \eqref{Property-Pell2} in Lemma~\ref{Lemma-Pell} to the last equation. 
\end{proof}

\begin{lemma}\label{Lemma-divisibility}
For every $q\in \mathbb{N}$, there exists $r=r(q)\in [2, \lceil\sqrt{2} q\rceil]_{\mathbb{Z}}$ such that $q\mid P_r$. 
\end{lemma}

\begin{proof}
Fix any $q\in \mathbb{N}$. In the case $q=1$, we can take $r=2$. Thus, we may assume $q\geq 2$.  
For every $P\in [0,\infty)_\mathbb{Z}$, there exist a unique pair of integers  $s$ and $t$ such that $P=qs+t$ and $0\leq t<q$. Then let $\overline{P} =t$. For every $n\in [0,\infty)_\mathbb{Z}$,
\[
(\overline {P_{n+1}}, \overline {P_n})\in  ([0,q-1]_\mathbb{Z})^2.
\]
The number of pairs $(n,m)$ of integers such that $0\leq n<m \leq \ceil{\sqrt{2}q} $ is equal to 
\[
\frac{1}{2}(\lceil \sqrt{2}q\rceil +1)  \lceil \sqrt{2}q\rceil > \frac{1}{2}\cdot 2 q^2 =q^2.
\]
By the pigeonhole principal, there exist integers $0\leq n<m \leq \lceil \sqrt{2}q\rceil$  such that 
\[
(\overline {P_{n+1}}, \overline {P_n})=(\overline {P_{m+1}}, \overline {P_m}).
\]
Therefore, by induction, we have $(\overline {P_{n+1+k}}, \overline {P_{n+k}})=(\overline {P_{m+1+k}}, \overline {P_{
m+k}}) $ for every $k\geq -n$.  By substituting $k=-n$, we obtain $(1,0)=(\overline {P_{1}}, \overline {P_0})=(\overline {P_{m+1-n}}, \overline {P_{m-n}})$. Set $r=m-n$. Then $r\in  [1,\ceil{\sqrt{2}q}]_{\mathbb{Z}}$. We deduce that $q \mid P_r$.  In addition, $q\geq 2$ and $q \mid P_r$ imply $r\geq 2$. Therefore, $r\in [2, \lceil\sqrt{2} q\rceil]_{\mathbb{Z}}$.
\end{proof}

\section{Proof of Theorem~\ref{Theorem-lowerT-Q} and Corollary~\ref{Corollary-lowerT-positive}}\label{Section-ProofTheoremCorollary}

\begin{proof}[Proof of Theorem~\ref{Theorem-lowerT-Q}]
Fix any $\alpha\in (0,1)\cap \mathbb{Q}$. Let $1\leq p<q$ be relatively prime integers satisfying $\alpha=p/q$. Let $r=r(q)\in [2,\ceil{\sqrt{2}q}]_{\mathbb{Z}}$ be as in Lemma~\ref{Lemma-divisibility}. It suffices to show that for every $n\in \mathbb{N}$,
\begin{equation}\label{Relation-rational}
\left(\left\lfloor \frac{p}{q}\: P_{2rn}^2 \right\rfloor,\ \left\lfloor \frac{p}{q}\: P_{2rn}^2 \right\rfloor,\   \left\lfloor \frac{p}{q} \:(P_{2rn}^2/2-1)^2 \right\rfloor \right)\in T(p/q).
\end{equation}
Indeed, by $r\geq 2$, we see that $P_{2rn}\geq  P_{4}=12$, and hence $P_{2rn}\leq P^2_{2rn}/2-1$. In addition, it follows from \eqref{Property-Pell3} in Lemma~\ref{Lemma-Pell} that for every $n\in \mathbb{N}$,
\begin{equation}\label{Inequality-Pphip}
 P_{2rn}\leq \frac{P_{2rn}^2}{2}-1 \leq \frac{1}{16}\cdot \phi_P^{4rn} \leq \frac{1}{16}\cdot \phi_P^{4\lceil\sqrt{2}q\rceil n}. 
\end{equation}
Further, we find that 
\begin{equation}\label{Inequality-EquivPhip}
\phi_P^{4\lceil \sqrt{2}q\rceil n}/16\leq x \iff n\leq \log(16x)/(4\lceil\sqrt{2}q \rceil\log\phi_P).
\end{equation}
Therefore, if \eqref{Relation-rational} is true, by combining \eqref{Inequality-Pphip} and \eqref{Inequality-EquivPhip}, we deduce that
\[
\# T_{\leq x}(p/q) \geq \floor{\log (16x)/ (4 \lceil \sqrt{2}q \rceil \log\phi_P)}
\]
for all $x>1$.

Let us show \eqref{Relation-rational}. Fix any $n\in \mathbb{N}$. Let $x=P_{2rn+1}-P_{2rn}$ and $y=P_{2rn}$. Then Lemma~\ref{Lemma-PellEquation} yields $x^2-2y^2=1$. 
By the choice of $r=r(q)$, we have $q\mid P_r$, and by \eqref{Property-Pell1} in Lemma~\ref{Lemma-Pell}, $P_r\mid P_{2rn}=y$ holds. Thus, we obtain $q\mid y$.
Hence, from $0<p/q<1$, 
\begin{equation}\label{equation-pell}
\left\lfloor  \frac{p}{q}\: x^2 \right\rfloor= \left\lfloor  \frac{p}{q}\: (2y^2+1) \right\rfloor=  2\cdot \frac{p}{q}\: y^2+\left\lfloor \frac{p}{q} \right\rfloor= 2\left\lfloor  \frac{p}{q}\: y^2 \right\rfloor. 
\end{equation}
In addition, for every $j\in \{-1,0,1\}$,
\[
\frac{p}{q}\: (y^2/2+j)^2 = \frac{p}{q}\: \frac{y^4}{4} + j \frac{p}{q}\: y^2 + \frac{p}{q} j^2. 
\]
By $q\mid y$ and $4\mid y^2$, $(p/q) (y^4/4)$ and $(p/q) y^2$ are integers. Further, $(p/q)j^2\in [0,1)$ for every $j\in \{-1,0,1\}$. Hence, for every $j\in \{-1,0,1\}$,
\begin{equation}\label{Equation-AP}
\left\lfloor \frac{p}{q}\: (y^2/2+j)^2 \right\rfloor= \left\lfloor \frac{p}{q}\: \frac{y^4}{4}\right\rfloor + j\left\lfloor  \frac{p}{q}\: y^2 \right\rfloor.
\end{equation}
Let $k=\ell= \lfloor (p/q)y^2 \rfloor$ and $m=\lfloor (p/q) (y^2/2-1)^2 \rfloor $.
By combining \eqref{equation-pell} and \eqref{Equation-AP}, we observe that $k,\ell,m\in S (p/q)$ and 
\begin{gather*}
    k+\ell=  \left\lfloor  \frac{p}{q}\: x^2 \right\rfloor \in S(p/q), \quad \ell+m=m+k=  \left\lfloor \frac{p}{q}\: (y^2/2)^2 \right\rfloor\in S(p/q), \\
    k+\ell +m= \left\lfloor \frac{p}{q}\: (y^2/2+1)^2 \right\rfloor \in S(p/q).
\end{gather*}
Therefore, $(k,\ell,m)=(k,k,m) \in T(p/q)$, which implies \eqref{Relation-rational}.
\end{proof}

\begin{lemma}\label{Lemma-IntegerParts}
Let $\alpha>0$ and $W>0$ be real numbers. Let $\delta=\delta(\alpha,W)= (1-\{ \alpha W\})/W$. Then for every $\beta\in [\alpha, \alpha+\delta)$, we have $\lfloor \alpha W \rfloor=\lfloor \beta W \rfloor$.
\end{lemma}

\begin{proof}
It follows that $\beta W= \alpha W + (\beta-\alpha)W =  \lfloor \alpha W \rfloor +\{ \alpha W\} + (\beta-\alpha)W$. Further, by the definitions of $\beta$ and $\delta$, we have
\[
0\leq \{ \alpha W\} + (\beta-\alpha)W< \{ \alpha W\} + \delta W=1,
\]
which yields that $\lfloor \beta W\rfloor =\lfloor \alpha W\rfloor$.  
\end{proof}

\begin{proof}[Proof of Corollary~\ref{Corollary-lowerT-positive}]
Let $A=\{\alpha \in [0,1] \colon \#T(\alpha)=\infty \} $. By Lemma~\ref{Lemma-IntegerParts} and \eqref{Relation-rational}, for every $\alpha\in A$, there exists a sufficiently small $\epsilon(\alpha)>0$ such that each $\gamma\in (\alpha,\alpha+\epsilon(\alpha))$ satisfies $\# T(\gamma)>0$. Then let 
\[
U=\left(\bigcup_{\alpha \in A} (\alpha, \alpha+\epsilon(\alpha)) \right)\cap (0,1).
\]
We find that $U$ is an open and dense set in $(0,1)$ since $\mathbb{Q}\subseteq A $ by Theorem~\ref{Theorem-lowerT-Q}. This implies Corollary~\ref{Corollary-lowerT-positive}.  
\end{proof}

\section{Uniform Distribution and Proof of Theorem~\ref{Theorem-lowerT-almostall}}\label{Section-UniformDistribution}

In this section, we address the theory of multi-dimensional uniform distributions and describe useful lemmas to prove Theorem~\ref{Theorem-lowerT-almostall}. Let $d\in \mathbb{N}$. For every $\mathbf{x}=(x_1,\ldots, x_d)\in \mathbb{R}^d$, we define
\[
\{\mathbf{x}\}=(\{x_1\},\ldots, \{x_d\}).  
\]
Let $(\mathbf{x}_{n})_{n\in \mathbb{N}}$ be a sequence of $\mathbb{R}^d$. We say that $(\mathbf{x}_{n})_{n\in \mathbb{N}}$ is \textit{uniformly distributed modulo $1$ on $[0,1)^d$} if for every $0\leq a_i<b_i\leq 1$ $(1\leq i\leq d)$, we have
\[
\lim_{N\to \infty}  \frac{\#\left\{ n\in [N] \colon \{\mathbf{x}_n\} \in \prod _{i=1}^d [a_i,b_i) \right\}}{N} =  \prod_{i=1}^d (b_i-a_i)  .
\]
We define $e(x)$ by $e^{2\pi \sqrt{-1} x }$ for all $x\in \mathbb{R}$, where $\sqrt{-1}$ denotes the imaginary unit. The following properties are equivalent:
\begin{enumerate}
\item \label{Property-UDM-1}  $(\mathbf{x}_{n})_{n\in \mathbb{N}}$ is uniformly distributed modulo $1$;
\item \label{Property-UDM-2} for all $\mathbf{h}\in \mathbb{Z}^d \setminus \{(0,\ldots, 0) \}$, we have
\[
\lim_{N\to \infty} \frac{1}{N} \sum_{n=1}^N e(\langle \mathbf{h}, \mathbf{x}_n \rangle ) =0,   
\]
where $\langle \cdot , \cdot \rangle$ denotes the standard inner product;
\item \label{Property-UDM-3} for every Jordan measurable set $X\subseteq [0,1)^d$ we have
\[
\lim_{N\to \infty}  \frac{\#\left\{ n\in [N] \colon \{\mathbf{x}_n\} \in X \right\}}{N} =  \mu_d(X),
\]
where $\mu_d$ denotes the Jordan measure on $\mathbb{R}^d$.    
\end{enumerate}
 The equivalence \eqref{Property-UDM-1} $\Leftrightarrow$ \eqref{Property-UDM-2} is referred to as  \textit{Weyl's criterion on} $[0,1)^d$.  
We refer the reader to \cite[Section~6 in Chapter~1]{KuipersNiederreiter} for more details on the theory of multi-dimensional uniform distribution of sequences.

\begin{lemma}\label{Lemma-almostall}
Let $A$ be a set, and $\mu$ be a measure on $A$. Let $(X_n)_{n\geq 1}$ be a bounded sequence of measurable functions defined on $A$. If the series
\[
\sum_{N\geq 1} \frac{1}{N} \int_A \left| \frac{1}{N} \sum_{1\leq n\leq N} X_n  \right|^2 \mathrm{d}\mu  
\]
converges, then $\mu$-almost all elements $\alpha$ of $A$ satisfy 
\[
 \lim_{N\to \infty} \frac{1}{N} \sum_{1\leq n\leq N} X_n (\alpha) =0.
\]
\end{lemma}
\begin{proof}
See \cite[Lemma~1.8]{Bugeaud}.
\end{proof}

 Let $X(N)$ and $Y(N)>0$ be quantities depending on $N$. We say that $X(N) \ll Y(N)$ if there exists a constant $C>0$ such that $|X(N)|\leq C\cdot Y(N)$. If this constant $C$ depends on some parameters $a_1,\ldots, a_n$, then we write $X(N) \ll_{a_1,\ldots,a_n} Y(N)$ for emphasizing the dependence.

\begin{lemma}\label{Lemma-uniformly}
Let $0<s<t<1$. The sequence $( (\alpha P_{2n}^4/4, \alpha P_{2n}^2))_{n\in \mathbb{N}}$ is uniformly distributed modulo $1$ for almost all $\alpha\in (s,t)$.
\end{lemma}

\begin{proof}
Fix any $(h_1,h_2)\in \mathbb{Z}^2 \setminus \{(0,0) \}$. Then for all $n,m\in \mathbb{N}$, we define 
\[
L(n,m) = (h_1/4) (P_n^4 -P_m^4)  + h_2 (P_n^2 -P_m^2).  
\]
When $h_1=0$, then $h_2 \neq 0$ by the choice of $(h_1,h_2)$, and hence
\[
| L(n,m) | =  |h_2| |P_n^2 -P_m^2|\geq  |P_n^2 -P_m^2|. 
\]
When $h_1 \neq 0$, then there exist $C_0=C_0(h_1,h_2)>0$ and $N_0=N_0(h_1,h_2)>0$ such that for all $n,m\geq N_0$, we obtain  
\begin{align}\label{estimation of L(n,m)}
|L(n,m)|
&\geq   |P_n^2 -P_m^2| (P_n^2+P_m^2) \left(|h_1|/4  -  |h_2|/(P_n^2+P_m^2) \right) \notag \\
&\geq  C_0 |P_n^2 -P_m^2 |.
\end{align}
Therefore, for sufficiently large $N\geq 1$, the following holds:
\begin{align}\nonumber
& \int_s^t  \left|\frac{1}{N} \sum_{n=1}^N  e( \alpha (h_1 P_{2n}^4/4+ h_2 P_{2n}^2) ) \right|^2\mathrm{d} \alpha \\ \nonumber
&= \frac{t-s}{N} + \frac{1}{N^2} \sum_{\substack {n,m\leq N \\ n\neq m}  } \int_{s}^t  e ( \alpha L(2n,2m) )\ \mathrm{d} \alpha\\ \label{Inequality-IntegrationE}
& \ll \frac{1}{N} + \frac{1}{N^2} \sum_{\substack {N_0\leq n,m\leq N \\ n\neq m}  } \frac{1}{|L(2n,2m)|}, 
\end{align}
where we apply the following to obtain the last inequality: 
\[
\int_{s}^t  e ( \alpha L(2n,2m) )\ \mathrm{d} \alpha=\frac{e(tL(2n,2m))-e(tL(2n,2m))}{2\pi L(2n,2m)\sqrt{-1}}\ll \frac{1}{|L(2n,2m)|}.
\]
Therefore, combining \eqref{estimation of L(n,m)} and \eqref{Inequality-IntegrationE},
\[
\int_s^t  \left|\frac{1}{N} \sum_{n=1}^N  e( \alpha (h_1 P_{2n}^4/4+ h_2 P_{2n}^2) ) \right|^2\mathrm{d} \alpha \ll \frac{1}{N} + \frac{1}{N^2} \sum_{\substack {N_0\leq m<n\leq N }  } \frac{1}{P_{2n}^2-P_{2m}^2}.
\]

Since $P_{2n}^2-P_{2m}^2 \geq P_{2n}^2 -P_{2n-2}^2 = (2P_{2n-1}+P_{2n-2})^2-P_{2n-2}^2 \geq  P_{2n-1}^2$, we obtain
\begin{align*}
\sum_{\substack {N_0\leq m<n\leq N }  } \frac{1}{P_{2n}^2-P_{2m}^2}
\ll N\sum_{\substack {N_0\leq n\leq N } } \frac{1}{P_{2n-1}^2}
\ll N\sum_{\substack {N_0\leq n }  } \phi_P^{-4n} \ll N,
\end{align*}
where $P_{2n-1}^2 \gg \phi_P^{4n}$ follows from \eqref{Property-Pell3} in Lemma~\ref{Lemma-Pell}. Therefore, the series
\[
\sum_{N\geq 1}\frac{1}{N}\int_s^t  \left|\frac{1}{N} \sum_{n=1}^N  e( \alpha (h_1 P_{2n}^4/4+ h_2 P_{2n}^2) ) \right|^2\mathrm{d} \alpha
\]
converges. By Lemma~\ref{Lemma-almostall} and Weyl's criterion, we conclude that
\[
( (\alpha P_{2n}^4/4, \alpha P_{2n}^2))_{n\in \mathbb{N}}
\]
is uniformly distributed modulo $1$ for almost all $\alpha\in (s,t)$.
\end{proof}

\begin{proof}[Proof of Theorem~\ref{Theorem-lowerT-almostall}]
Let $0<s<t<1$. Lemma~\ref{Lemma-uniformly} establishes that \\
$((\alpha P_{2n}^4/4 ,\alpha P_{2n}^2 ))_{n\in \mathbb{N}}$ is uniformly distributed modulo $1$ on $[0,1)^2$ for almost all $\alpha\in (s,t)$. Take such $\alpha\in (s,t)$. Define 
\[
A=\{n\in \mathbb{N} \colon (\{\alpha P_{2n}^4/4 \},\{\alpha P_{2n}^2\} )\in [a,(1-t)/2]\times [0,(1-t)/2] \}, 
\]
where $a:=\max(0, (1-t)/2-s)$. Fix any $n\in A$. Set $x_n=P_{2n+1}-P_{2n}$ and $y_n=P_{2n}$. Then Lemma~\ref{Lemma-PellEquation} yields 
\[
 \alpha x_n^2= \alpha (2y_n^2 +1) = 2 \lfloor \alpha y_n^2 \rfloor +2\{ \alpha y_n^2\} +\alpha. 
\]
Since $0\leq 2\{ \alpha y_n^2\} +\alpha < (1-t) +t=1$, we have 
\begin{equation}\label{Equation-Theorem1.4-xy}
\lfloor \alpha x_n^2 \rfloor =2\lfloor \alpha y_n^2 \rfloor.  
\end{equation}
Further, for every $j\in \{-1,0,1\}$, we have
\[
 \alpha(y_n^2/2+j)^2 = \alpha y_n^4/4 +j\alpha y_n^2 +\alpha j^2
 = \lfloor \alpha y_n^4/4 \rfloor + j \lfloor  \alpha y_n^2 \rfloor + \delta (j),  
\]
where $\delta(j)=\{  \alpha y_n^4/4\} + j\{ \alpha y_n^2\}+\alpha j^2$.
By the choice of $n$, we obtain 
\[
0\leq a- (1-t)/2+ s < \delta(j) <  (1-t)/2+(1-t)/2+t =1
\]
for every $j\in \{-1,0,1\}$. Therefore, each $n\in A$ satisfies 
\begin{equation}\label{Equation-Theorem1.4-AP}
\lfloor \alpha(y_n^2/2+j)^2 \rfloor=\lfloor \alpha y_n^4/4 \rfloor + j \lfloor  \alpha y_n^2 \rfloor.
\end{equation}
Combining \eqref{Equation-Theorem1.4-xy} and \eqref{Equation-Theorem1.4-AP}, each $n\in A$ meets the condition that 
\[
(\floor{ \alpha y_n^2} , \floor{\alpha y_n^2}, \floor{\alpha (y_n^2/2-1)^2})\in T(\alpha).
\]
It follows that $y_n^2/2-1=P_{2n}^2/2 -1\leq \phi_P^{4n}/16$. Thus, we have
\[
\#T_{\leq x} (\alpha) \geq \# \left(A \cap \left[1,  \frac{\log (16x)} {4 \log \phi_P} \right] \right).
\]
By the choice of $\alpha$, $((\alpha P_{2n}^4/4 ,\alpha P_{2n}^2 ))_{n\in \mathbb{N}}$ is uniformly distributed modulo $1$. This yields 
\begin{align*}
&\liminf_{x\to \infty}\frac{4\log (\phi_P)} {\log (16x)}\cdot\#T_{\leq x} (\alpha) \geq \liminf_{x\to \infty}\ \frac{4\log (\phi_P)} {\log (16x)}\cdot \# \left(A \cap \left[1,  \frac{\log (16x)} {4 \log \phi_P} \right] \right)\\
&= \frac{1-t}{2} \left(\frac{1-t}{2}-a\right) = \frac{ (1-t) \cdot \min \left( 1-t,\ 2s \right)}{4}.
\end{align*}
Therefore, we deduce the conclusion of Theorem~\ref{Theorem-lowerT-almostall}.
\end{proof}

\section{Proof of Theorem~\ref{Theorem-Upperbounds}}\label{Section-Upperbounds}
The goal of this section is to discuss upper bounds for $\#T_{\leq x}(\alpha)$ and to provide a proof of Theorem~\ref{Theorem-Upperbounds}. Let $a_1$, $a_2$, and $a_3$ be integers. For every $x> 1$, we define $\mathcal{A}(x)=\mathcal{A}(x,a_1,a_2,a_3)$ as the set of tuples $(k,\ell, m)\in [1,x]_\mathbb{Z}$ such that 
\[
k^2+\ell^2 =y^2 +a_1,\quad k^2 +m^2=z^2 +a_2,\quad m^2 +\ell^2=w^2 +a_3  
\]
for some $y,z,w\in \mathbb{N}$. We observe that 
\begin{equation}\label{Inequality-TtoA}
\# T_{\leq x} (\alpha) \ll \sum_{|a_1|, |a_2|, |a_3|\leq 2/\alpha,  } \#\mathcal{A}(x,a_1,a_2,a_3) .
\end{equation}
Indeed, if $(\lfloor \alpha k^2 \rfloor ,\lfloor \alpha \ell^2\rfloor,\lfloor \alpha m^2\rfloor )\in T_{\leq x}(\alpha)$, then we have 
\[
\lfloor \alpha k^2\rfloor +\lfloor \alpha \ell^2\rfloor =\lfloor \alpha y^2 \rfloor,\quad \lfloor \alpha k^2\rfloor +\lfloor \alpha m^2\rfloor =\lfloor \alpha z^2 \rfloor,\quad 
\lfloor \alpha m^2\rfloor +\lfloor \alpha \ell^2\rfloor =\lfloor \alpha w^2 \rfloor
\]
for some $y,z, w\in \mathbb{N}$, which is equivalent to
\begin{align*}
\alpha(k^2+l^2-y^2)&=\{\alpha k^2\}+\{\alpha l^2\}-\{\alpha y^2\},\\
\alpha(k^2+m^2-z^2)&=\{\alpha k^2\}+\{\alpha m^2\}-\{\alpha z^2\},\\
\alpha(m^2+l^2-w^2)&=\{\alpha m^2\}+\{\alpha l^2\}-\{\alpha w^2\}.
\end{align*}
Therefore, 
\[
|k^2 +\ell^2 -y^2| \leq 2/\alpha, \quad |k^2 +m^2 -z^2| \leq 2/\alpha, \quad  |m^2 +\ell^2-w^2|\leq 2/\alpha .
\]
This implies that each tuple $(k,\ell , m)\in T_{\leq x}(\alpha)$ belongs to $\mathcal{A}(x,a_1,a_2,a_3)$ for some $|a_1|,|a_2|,|a_3|\leq 2/\alpha$. Thus, we deduce \eqref{Inequality-TtoA}.

From \eqref{Inequality-TtoA}, it suffices to show that for every $x\geq 2$
\begin{equation}\label{Inequality-Atoxlogx}
\#\mathcal{A}(x,a_1,a_2,a_3) \ll_{a_1, a_2,a_3}   x(\log x)^{15}.
\end{equation}
We can apply the divisor function to evaluate upper bounds for $\#\mathcal{A}$. For all $n\in \mathbb{Z}\setminus \{0\}$, we define $\tau (n)$ as the number of negative or positive divisors of $n$, and $\tau_+(n)$ as the number of positive divisors of $n$. It is clear that $\tau(n)=2\tau_+(n)$ since for every divisor $q$ of $n$, both $q$ and $-q$ are negative or positive divisors of $n$. By the fundamental theorem of arithmetic,  for all $u,v\in \mathbb{N}$, 
\begin{equation} \label{Inequality-divisor}
\tau_+(uv) \leq \tau_+(u) \tau_+(v). 
\end{equation}

By the definition of $\mathcal{A}$, for every fixed $k\in [1,x]_\mathbb{Z}$, if $(k,\ell,m)\in \mathcal{A}(x,a_1,a_2,a_3)$, then there exist $y,z\in \mathbb{N}$ such that 
\begin{equation}\label{Equation-YLZM}
(y-\ell)(y+\ell) = k^2-a_1 ,\quad (z-m)(z+m)=k^2-a_2.
\end{equation}
Therefore, if $k^2-a_1\neq 0$ and $k^2-a_2\neq 0$, then the number of tuples $(y,\ell, z,m)$ satisfying \eqref{Equation-YLZM} is less than or equal to
\[
\tau (k^2-a_1) \tau (k^2-a_2).
\]
Thus, roughly speaking, we obtain 
\begin{equation}\label{Inequality-AtoTau}
\#\mathcal{A}(x,a_1,a_2,a_3)\ll \sum_{\substack{k\leq x\\ k^2\notin \{a_1,a_2\}}} \tau (k^2-a_1) \tau (k^2-a_2).
\end{equation}
Note that we have to care about the case $k^2\in \{a_1,a_2\}$ as well. We will do precise discussion at the end of this section.   
To evaluate the right-hand side in \eqref{Inequality-AtoTau}, the following lemma is useful.

\begin{lemma}\label{Lemma-MeanIrreducible} Let $s\in \mathbb{N}$, and let $P\in \mathbb{Z}[x]$ be an irreducible polynomial of degree $\ell\in\mathbb{N}$. Assume that $P(x)>0$ for all $x\in \mathbb{N}$. Then for all $x\geq 2$, 
\begin{equation}\label{Inequality-MeanIrreducible}
 x(\log x)^{ 2^{s} -1 }\ll_{P,s} \sum_{m\leq x} \tau_+^s(P(m)) \ll_{P,s} x(\log x)^{ 2^{s} -1 }.
\end{equation}
\end{lemma}

\begin{proof}
See \cite[Theorem]{Delmer}.
\end{proof}

\begin{lemma}\label{Lemma-Quad} Let $s\in\mathbb{N}$, and let $P\in \mathbb{Z}[x]$ be a polynomial of degree $2$. Assume that there exists $x_0=x_0(P)>0$ such that $P(x)>0$ for all $x\geq x_0$. Then for all $x\geq 2$, 
\[
\sum_{1\leq m\leq x,\: P(m)\neq 0} \tau^s(P(m)) \ll_{P,s} x(\log x)^{ 4^s-1 }.
\]
\end{lemma}

\begin{proof}
If $P\in \mathbb{Z}[x]$ is irreducible, then it is clear from Lemma~\ref{Lemma-MeanIrreducible} that  
\[
\sum_{m\leq x,\: P(m)\neq 0} \tau^s(P(m)) \ll_{P,s} 1+\sum_{x_0\leq m\leq x\: } \tau_+^s(P(m))\ll_{P,s} x(\log x)^{ 2^{s} -1 }.
\]
If $P$ is reducible, then let $P=P_1P_2$, where $P_1,P_2\in \mathbb{Z}[x]$ are of degree $1$. In this case, by \eqref{Inequality-divisor} and the Cauchy-Schwarz inequality, we obtain
\begin{align*}
\sum_{x_0\leq m\leq x} \tau_+^s(P(m))&\leq \sum_{x_0\leq m\leq x} \tau_+^s(P_1(m)) \tau_+^s(P_2(m))\\
&\leq \left(\sum_{x_0\leq m\leq x} \tau_+^{2s}(P_1(m))\right)^{1/2}
\left(\sum_{x_0\leq m\leq x} \tau_+^{2s}(P_2(m))\right)^{1/2}.
\end{align*}
Therefore, applying Lemma~\ref{Lemma-MeanIrreducible},
\begin{align*}
\sum_{x_0\leq m\leq x} \tau_+^s(P(m))&\ll_{P,s}  x (\log x)^{4^{s}-1}.
\end{align*}

\end{proof}

\begin{proof}[Proof of \eqref{Inequality-Atoxlogx}]
Let $x\geq 2$ be a real number. For all conditions $P$ in which $k,\ell, m$ appear as free variables, we define
$\mathcal{A}(P)=\mathcal{A}(P, x, a_1,a_2,a_3)$ as the set of all $(k,\ell, m)\in \mathcal{A}(x,a_1,a_2,a_3)$ satisfying that $P$ is true. 
Let $(k, \ell, m)\in \mathcal{A}$. Then there exist $y, z, w\in \mathbb{N}$ such that 
\begin{align}\label{Equation-divisorYL}
(y-\ell)(y+\ell) &= k^2-a_1,\\ \label{Equation-divisorZM}
(z-m)(z+m)&=k^2-a_2,\\ \label{Equation-divisorWL}
(w-\ell)(w+\ell)&=m^2-a_3.    
\end{align}

\noindent\underline{\textbf{Case~1:}} We now discuss the case when $k^2\in \{a_1,a_2 \}$ and $m^2=a_3$. In this case, it is clear from $1\leq \ell \leq x$ that 
\[
\#\mathcal{A} (k^2\in \{a_1,a_2 \}\: \&\: m^2=a_3)\ll x.
\]

\noindent\underline{\textbf{Case~2:}} We discuss the case when $k^2\in \{a_1,a_2 \}$ and $m^2\neq a_3$.  For every fixed $m\leq x$ with $m^2 \neq a_3$, the number of $(w,\ell)$ satisfying   \eqref{Equation-divisorWL} is less than or equal to $\tau(m^2-a_3)$. Therefore, Lemma~\ref{Lemma-Quad} yields 
\[
\#\mathcal{A} (k^2\in \{a_1,a_2 \}\: \&\: m^2\neq a_3)\leq 
 \sum_{m\leq x,\: m^2\neq a_3}  \tau \left(m^2 -a_3 \right)\ll_{a_3} x(\log x)^3.
\]

\noindent \underline{\textbf{Case~3:}} We discuss the case when $k^2\notin\{a_1,a_2\}$. By \eqref{Equation-divisorYL} and \eqref{Equation-divisorZM}, for every $(k,\ell ,m)\in \mathcal{A}(k^2\notin \{a_1,a_2\})$ there exist $y,z\in \mathbb{N}$ such that
\[
(y-\ell)(y+\ell) = k^2-a_1\neq 0 ,\quad (z-m)(z+m)=k^2-a_2 \neq 0.
\]
Therefore, it follows that
\begin{align*}
\# \mathcal{A}(k^2\notin \{a_1,a_2\})&\leq  \sum_{k\leq x,\: k^2\notin \{a_1, a_2\}}  \tau\left(k^2 -a_2 \right)\tau\left(k^2-a_1 \right).
\end{align*}
The Cauchy-Schwarz inequality and Lemma~\ref{Lemma-Quad} imply that
\begin{align*}
\# \mathcal{A}(k^2\notin \{a_1,a_2\})&\leq  \left(\sum_{k\leq x} \tau^2(k^2-a_2) \right)^{1/2} \left(\sum_{k\leq x} \tau^2(k^2-a_1) \right)^{1/2}\\
&\ll x(\log x)^{15}. 
\end{align*}
We deduce \eqref{Inequality-Atoxlogx}, and hence we obtain Theorem~\ref{Theorem-Upperbounds}.
\end{proof}

\section{Homogeneous Equations and Proof of Theorem~\ref{Theorem-infinitePEBs}}\label{Section-Homogeneous}

In this section, we firstly present a generalization of the floor function and ceiling function. Let $\bbracket{\:\cdot\:}: \mathbb{R} \to \mathbb{Z}$ be an arbitrary function satisfying the following conditions: 
setting $\eta(x)= |x-\bbracket{x}|$,  
\begin{enumerate}
    \item[(C1)] for all $x\in \mathbb{R}$, $\eta(x)=\eta(x+1)$;
    \item[(C2)] for all $x\in \mathbb{Z}$, $\eta(x)=0$; 
    \item[(C3)] for all $x\in \mathbb{R}$ and $n\in \mathbb{N}$, $\eta(nx)\leq n\eta (x) $;
    \item[(C4)] for all $t>0$, $\eta^{-1}([0,t])\cap [0,1)$ is Jordan measurable and has positive measure.   
\end{enumerate}
We can give the floor function or the ceiling function as an example of $\bbracket{\:\cdot\:}$. Indeed, setting $\eta (x):=|x-\lfloor x\rfloor|= \{x\}$, $\eta$ satisfies (C1) to (C4). Similarly, $\eta(x):=|x-\lceil x \rceil|$ also satisfies (C1) to (C4).  In addition, for every $x\in \mathbb{R}$, let $\|x\|$ be the closest integer to $x$. We can also take $\|\cdot \|$ as $\bbracket{\:\cdot\:}$.

We can generalize the Diophantine equations on $T(\alpha)$. Let $f:\mathbb{R}^r \to \mathbb{R}$ be a homogeneous polynomial with coefficients in $\mathbb{Z}$. Let $d$ be the degree of $f$. We can write
\[
f (x_1,\ldots , x_r) = \sum_{ \substack{\nu_1,\ldots,\nu_r\geq 0\\  \nu_1+\cdots +\nu_r=d}   } a_{\nu_1,\ldots ,\nu_r} x_1^{\nu_1} \cdots x_r^{\nu_r}  
\]
for some $a_{\nu_1,\ldots ,\nu_r}\in \mathbb{Z}$. Then for all $x_1,\ldots, x_r \in \mathbb{R}^r$ and for all $\alpha\in \mathbb{R}$, we define
\[
\bbracket{\alpha f} (x_1,\ldots , x_r) = \sum_{ \substack{\nu_1,\ldots,\nu_r\geq 0\\  \nu_1+\cdots +\nu_r=d}   } a_{\nu_1,\ldots ,\nu_r} \bbracket{\alpha x_1^{\nu_1} \cdots x_r^{\nu_r}}.
\]
The goal of this section is to develop the following theorem.
\begin{theorem}\label{Theorem-homogeneous}
 Let $K\in \mathbb{N}$. For every $j\in [K]$, let $f_j:\mathbb{R}^r \to \mathbb{R}$ be a non-zero homogeneous polynomial with coefficients in $\mathbb{Z}$. Assume that there exists $(x_1,\ldots ,x_r)\in \mathbb{N}^r$ such that for all $j\in [K]$, we have $f_j(x_1,\ldots , x_r)=0$. Then for all $\alpha\in \mathbb{R}$,  
 \[
\liminf_{N\to \infty} \frac{1}{N} \#\{n\in [N] \colon \bbracket{\alpha f_j} (nx_1,\ldots , nx_r)=0 \text{ for all $j\in [K]$} \}>0.
\]
\end{theorem}

\begin{proof}[Proof of Theorem~\ref{Theorem-infinitePEBs} assuming Theorem~\ref{Theorem-homogeneous}] Assume that there is a perfect Euler brick. Fix any $\alpha>0$. Then there exist $k, \ell, m, a, b, c, d\in \mathbb{N}$ with $k\leq \ell\leq m$ such that 
\begin{gather*}
    k^2+\ell^2 =a^2,\quad \ell^2+m^2=b^2,\quad m^2+k^2=c^2,\quad k^2+\ell^2+m^2=d^2.
\end{gather*}
Let $A$ be the set of all $n\in \mathbb{N}$ such that
\begin{gather*}
    \floor{\alpha (nk)^2}+\floor{\alpha (n\ell)^2} =\floor{\alpha (na)^2},\\
    \floor{\alpha (n\ell)^2}+\floor{ \alpha( nm)^2}=\floor{ \alpha (nb)^2},\\  
    \floor{ \alpha (nm)^2}+\floor{ \alpha (nk)^2}=\floor{ \alpha (nc)^2},\\
    \floor{ \alpha (nk)^2}+\floor{ \alpha (n\ell)^2}+\floor{ \alpha (nm)^2}=\floor{ \alpha (nd)^2}.
\end{gather*}
Note that each $n\in A$ satisfies $(\floor{\alpha (nk)^2},\floor{\alpha (n\ell)^2}, \floor{ \alpha( nm)^2})\in T(\alpha)$. By Theorem~\ref{Theorem-homogeneous} with $\bbracket{\:\cdot\:}=\floor{\:\cdot\:}$, we have
\[
\liminf_{x\to \infty} \frac{\#T_{\leq x}(\alpha)}{x} \geq \liminf_{x\to \infty}\frac{1}{m} \frac{\#(A\cap [1, x/m ]) }{x/m}>0.
\]
\end{proof}


\begin{lemma}\label{Lemma-irrational-distribution}
Let $p(x) = c_mx^m + \cdots + c_1x+c_0, m\geq 1$, be a polynomial with real coefficients and let at least one of the coefficients $c_j$ with $j > 0$ be irrational. Then the sequence $(p(n))_{\in \mathbb{N}}$ is uniformly distributed modulo $1$.
\end{lemma}
\begin{proof}
See \cite[Theorem~3.2]{KuipersNiederreiter}.
\end{proof}

\begin{lemma}\label{Lemma-joint-distribution}
Let $d_1,\ldots, d_s$ be pairwise distinct positive integers. Let $\alpha\in \mathbb{R}$ be irrational. Then $((\alpha n^{d_1}, \ldots , \alpha n^{d_s}) )_{n=1}^\infty$ is uniformly distributed modulo $1$ on $[0,1)^s$.
\end{lemma}

\begin{proof}
Fix any $\mathbf{h}=(h_1,\ldots ,h_s)\in \mathbb{Z}\setminus \{0\}$. Then $(h_1\alpha n^{d_1}+\cdots + h_s \alpha n^{d_s} )_{n=1}^\infty$ is uniformly distributed modulo $1$ on $[0,1)$ by Lemma~\ref{Lemma-irrational-distribution}. Therefore, by Weyl's criterion on $[0,1)$,
\[
\lim_{N\to \infty} \frac{1}{N} \sum_{n=1}^N e(h_1\alpha n^{d_1}+\cdots + h_s \alpha n^{d_s}  ) =0.  
\]
This implies that $((\alpha n^{d_1}, \ldots , \alpha n^{d_s}))_{n=1}^\infty$ is uniformly distributed modulo $1$ by applying Weyl's criterion on $[0,1)^s$.
\end{proof}

\begin{proof}[Proof of Theorem~\ref{Theorem-homogeneous}]
By the assumption, we fix a tuple $(x_1,\ldots ,x_r)\in \mathbb{N}^r$ such that for all $j\in [K]$, we have $f_j(x_1,\ldots , x_r)=0$. Fix any $\alpha\in \mathbb{R}$. Let $d_j$ be the degree of $f_j$. Then for all $j\in [K]$ and $n\in \mathbb{N}$, 
\begin{align*}
|\bbracket{\alpha f_j} (nx_1,\ldots , nx_r)|&\leq |\alpha f_j (nx_1,\ldots , nx_r)- \bbracket{\alpha f_j} (nx_1,\ldots , nx_r)|\\
&\leq C \sum_{\substack{\nu_1,\ldots,\nu_r\geq 0\\  \nu_1+\cdots +\nu_r=d_j} }  \eta ( \alpha n^{d_j} x_1^{\nu_1} \cdots x_r^{\nu_r}) 
\end{align*}
for some constant $C>0$ that depends only on $f_1,\ldots ,f_r$ and $K$. Let $P$ be the maximum number of tuples $(\nu_1,\ldots, \nu_r)\in \mathbb{Z}^r$ such that $\nu_1,\ldots,\nu_r\geq 0$ and $\nu_1+\cdots +\nu_r=d_j$ for  $j\in [K]$. 
Let 
\begin{align*}
A=\{n\in \mathbb{N} \colon  & 0\leq \eta (\alpha n^{d_j} x_1^{\nu_1} \cdots x_r^{\nu_r})< (CP)^{-1}   \text{ for all } j\in [K] \\
& \text{and for all } \nu_1,\ldots,\nu_r\geq 0 \text{ with }  \nu_1+\cdots +\nu_r=d_j   \}.
\end{align*}
Take any $n\in A$. We have $|\bbracket{\alpha f_j} ( nx_1,\ldots , nx_r)|<1$ for every $j\in [N]$. Hence, we deduce
\[
\bbracket{\alpha f_j} (nx_1,\ldots , nx_r)=0
\]
for every $j\in [K]$, since $\bbracket{\alpha f_j} (nx_1,\ldots , nx_r)\in \mathbb{N}$. Therefore, it is sufficient to show that 
\begin{equation}\label{Inequality-liminf-A}
\liminf_{N\to \infty} \#(A\cap [N])/N>0. 
\end{equation}
When $\alpha$ is rational, there exist integers $p,q$ such that $\alpha=p/q$. By using (C2), substituting $n=q, 2q, 3q,\ldots$ yields  $\eta (\alpha n^{d_j} x_1^{\nu_1} \cdots x_r^{\nu_r})=0$. Therefore, \eqref{Inequality-liminf-A} holds. In the case where $\alpha$ is irrational, if necessary, by renumbering $d_1,\ldots, d_K$, we may assume that there exist $ 1\leq j_1<j_2<\cdots <j_s=K$ such that
\[
1\leq d_1= \cdots =d_{j_1} < d_{j_1+1}=\cdots = d_{j_2} < \cdots < d_{j_{s-1}+1}=\cdots =d_{j_s}=d_K.   
\]
Let $M=(x_1\cdots x_r)^{d_K}$, and define
\[
B=\Pas{n\in \mathbb{N} \colon 0\leq \eta\left(\alpha n^{d_{j_u} }\right)\leq (2CP M)^{-1} \text{ for all } u\in [s]  }.  
\]
From the irrationality of $\alpha$, Lemma~\ref{Lemma-joint-distribution} shows that $((\alpha n^{d_{j_1}}, \ldots, \alpha n^{d_{j_s}})  )_{n=1}^\infty$ is uniformly distributed modulo $1$ on $[0,1)^s$. Therefore, by (C1) and (C4), we obtain
\begin{align*}
 &\frac{1}{N}\#(B\cap [1,N]) \\
 &= \frac{1}{N} \#\Pas{n\in [N] \colon \Pas{(\alpha n^{d_{j_1}},\ldots,\alpha n^{d_{j_s}} )}  \in \left(\eta^{-1}\pas{ \left[0, \frac{1}{2CPM}\right] } \cap [0,1)\right)^s   } \\
 &\to \mu_s \left(\pas{ \eta^{-1}\pas{ \left[0, \frac{1}{2CPM}\right] } \cap [0,1)}^s\: \right)>0 \quad \text{(as $N\to \infty$) }.
\end{align*}
Thus, $B$ has positive density. Further, by (C3), each $n\in B$ satisfies the condition that for all  $j\in [K]$ and integers $\nu_1,\ldots,\nu_r\geq 0$ \text{ with }  $\nu_1+\cdots +\nu_r=d_j$,
\begin{align*}
\eta (\alpha n^{d_j} x_1^{\nu_1} \cdots x_r^{\nu_r}) \leq x_1^{\nu_1} \cdots x_r^{\nu_r} \eta (\alpha n^{d_j} ) \leq M (2CPM)^{-1} <(CP)^{-1}.
\end{align*}
Therefore, $B\subseteq A$, which yields \eqref{Inequality-liminf-A}. 
\end{proof}

\section{Further discussions}\label{Section-FurtherDiscussion}
From Corollary~\ref{Corollary-lowerT-positive}, Theorem~\ref{Theorem-lowerT-almostall}, and Theorem~\ref{Theorem-infinitePEBs}, we pose the following question.
\begin{question}\label{Problem-zeropoint}
Is there $\alpha\in(0,1)$ such that $T(\alpha)$ is empty?
\end{question}
By Theorem~\ref{Theorem-infinitePEBs}, if we were to obtain an affirmative answer to this question, then there would be no perfect Euler brick. Note that $T(\alpha)$ is non-empty for all $\alpha\in(0,1/9)$. Indeed,  it is enough to show that $1,2,3\in S(\alpha)$. For every $\alpha\in(0,7-4\sqrt{3}]$,
	\[
	1\leq \sqrt{\frac{4}{\alpha}}-\sqrt{\frac{3}{\alpha}}\leq\sqrt{\frac{3}{\alpha}}-\sqrt{\frac{2}{\alpha}}\leq\sqrt{\frac{2}{\alpha}}-\frac{1}{\sqrt{\alpha}}.
	\]
	Thus, the intervals $[1/\sqrt{\alpha},\sqrt{2/\alpha}), [\sqrt{2/\alpha},\sqrt{3/\alpha}), [\sqrt{3/\alpha},\sqrt{4/\alpha})$ contain some integers. This implies that there exist $n_1,n_2,n_3\in \mathbb{N}$ such that $\floor{\alpha n_1^2}=1$, $\floor{\alpha n_2^2}=2$, and $\floor{\alpha n_3^2}=3$. In addition, we obtain $\floor{16\alpha}=1$, $\floor{36\alpha}=2$, and $\floor{49\alpha}=3$ for all $\alpha\in [1/16,4/49)$. We also find that $\floor{36\alpha}=2$, $\floor{49\alpha}=4$,  and $\floor{81\alpha}=6$ for all $\alpha\in [4/49,1/12)$;  $\floor{16\alpha}=1$, $\floor{25\alpha}=2$, and $\floor{36\alpha}=3$ for all $\alpha\in [1/12,1/9)$. Therefore, $T(\alpha)$ is non-empty for all $\alpha\in (0,1/9)$.
\begin{conjecture}\label{Conjecture-numberofTx(alpha)}
 For suitable $\alpha\in (0,1)$, there exists $\lambda=\lambda(\alpha)>0$ such that 
 \[
 \# T_{\leq x}(\alpha)=\lambda \log x +o(\log x) \quad (\text{as } x\to \infty).
 \]
 \end{conjecture}
 Note that if the conjecture were true for some $\alpha\in (0,1)$, then there would be no perfect Euler brick by Theorem~\ref{Theorem-infinitePEBs}. To support this conjecture, we show the graphs of $\# T_{\leq x}(0.1)$ and $\# T_{\leq x}(0.2)$ for $0\leq x\leq 46300$ in Figure~\ref{Figure1}. 
 
 \begin{figure}[htbp]
 \caption{}\label{Figure1}
  \begin{minipage}[t]{0.48\linewidth}
    \centering
    \includegraphics[keepaspectratio, scale=0.3]{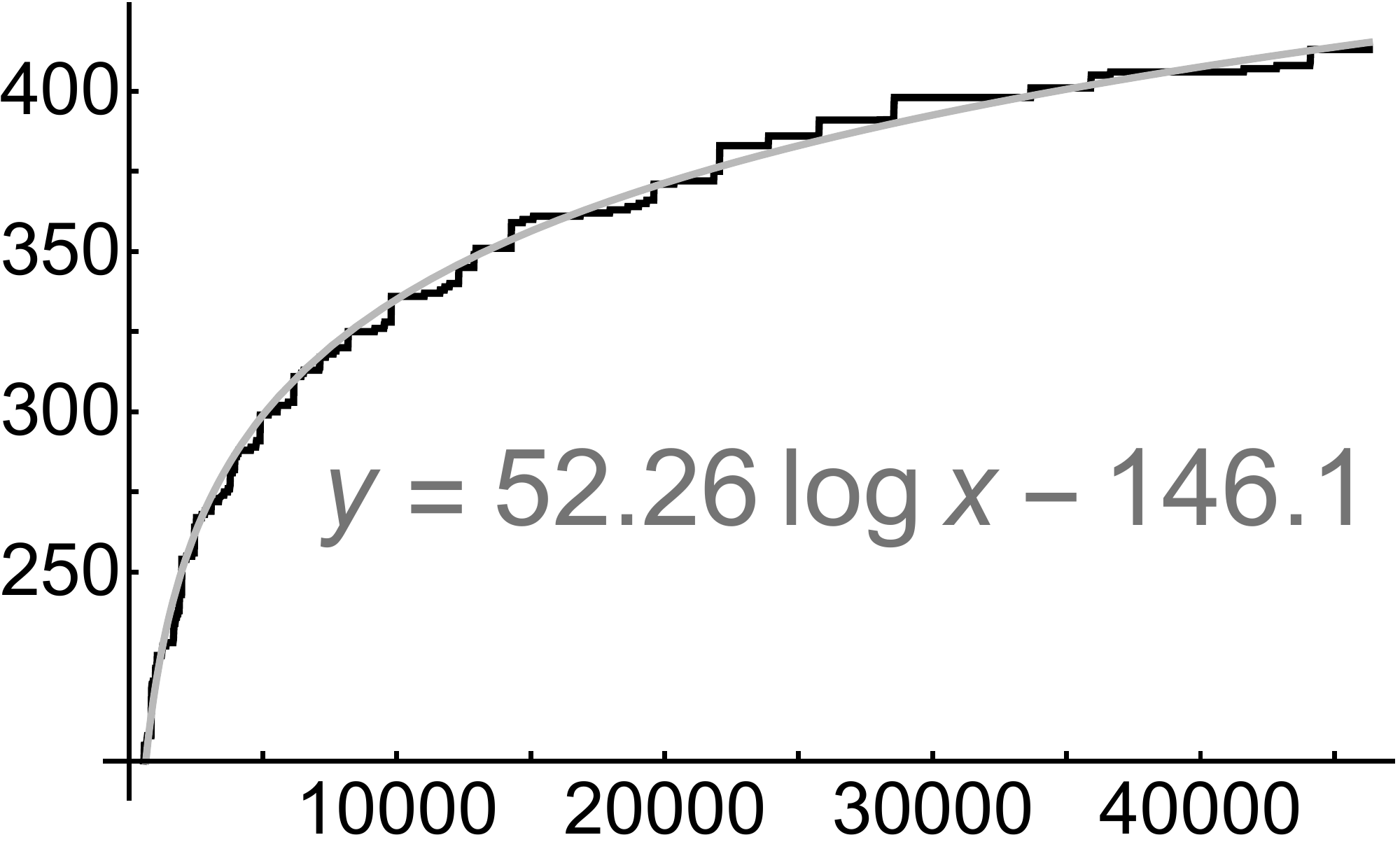}
    \caption*{$\alpha=0.1,\ x\leq 46300$}
  \end{minipage}
  \begin{minipage}[t]{0.48\linewidth}
    \centering
    \includegraphics[keepaspectratio, scale=0.3]{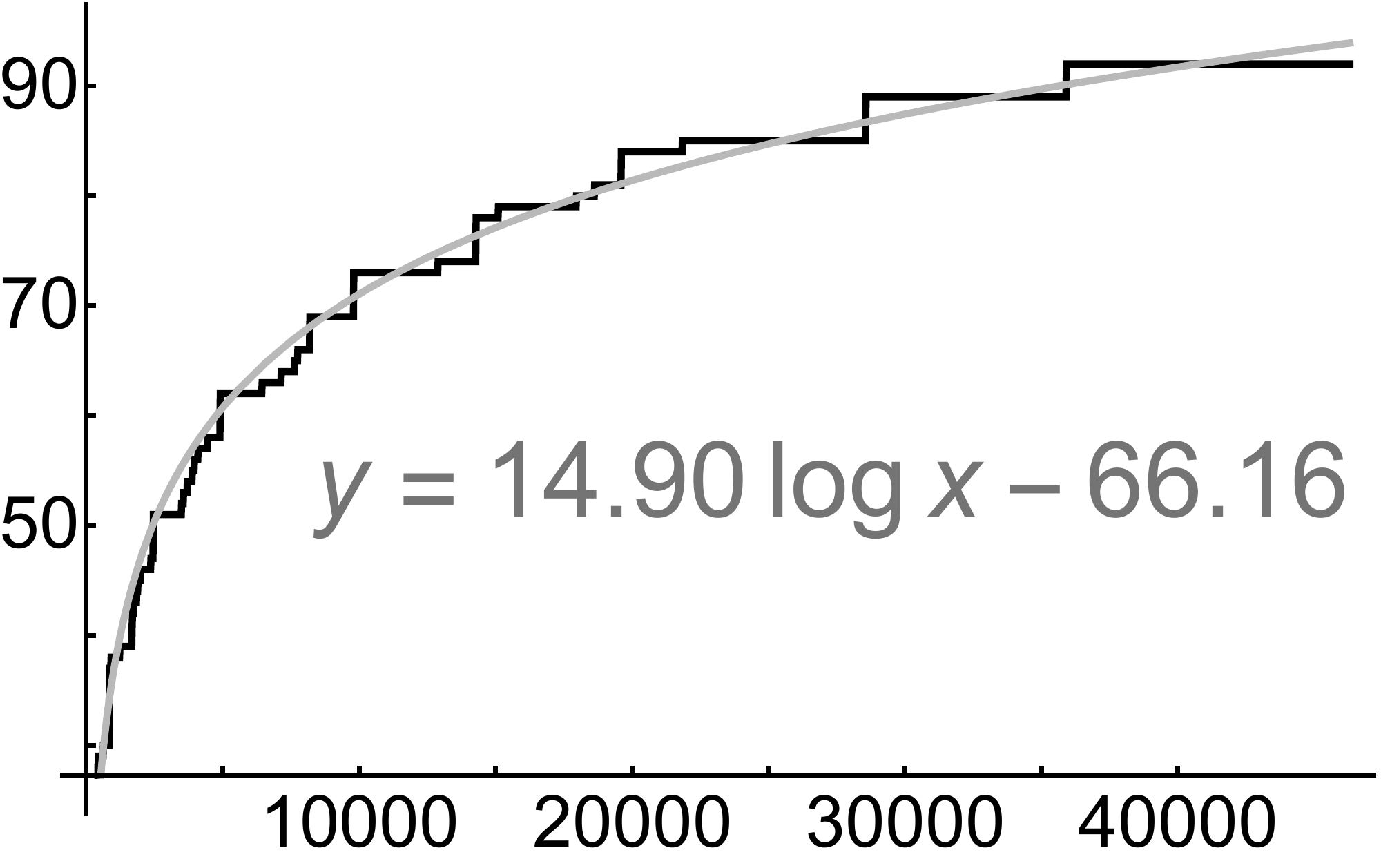}
    \caption*{$\alpha=0.2,\ x\leq 46300$}
  \end{minipage}
\end{figure}
 The step functions in the figure are $y=\#T_{\leq x}(\alpha)$ for $\alpha \in \{0.1, 0.2\}$. Interestingly, we observe that the step functions are well approximated by the functions of the form $y=\lambda \log x +\kappa$. Hence, these numerical calculations support that Conjecture~\ref{Conjecture-numberofTx(alpha)} is true. In addition, Theorem~\ref{Theorem-lowerT-Q} and  Theorem~\ref{Theorem-lowerT-almostall} provide theoretical support. Indeed, these theorems show that for some suitable $\alpha\in (0,1)$, there exists $\lambda'=\lambda'(\alpha)>0$ such that for sufficiently large $x>1$, we have $\# T_{\leq x} (\alpha) \geq \lambda' \log x$. It should be noted  that there would be large gaps between $\lambda'$ and $\lambda$, since Theorem~\ref{Theorem-lowerT-Q} shows that
 \[
 \# T_{\leq x} (0.1) \geq \floor{\frac{\log(16x) }{4\ceil{ \sqrt{2}\cdot 10 } \log (1+\sqrt{2}) } },
 \]
 and hence $\lambda'<(4\ceil{\sqrt{2}\cdot 10}\log (1+\sqrt{2}))^{-1}=0.018909\cdots$ is much less than $52.26$. 
 
 In the case of $\alpha\in \{0.3,0.4\}$, the number of $T_{\leq x}(\alpha)$ also appears to be of the form $\lambda \log x +\kappa$; however, the number is too small to allow us to infer whether the conjecture would be true (Figure~\ref{Figure2}). 
 
 \begin{figure}[htbp]
 \caption{}\label{Figure2}
  \begin{minipage}[t]{0.48\linewidth}
    \centering
    \includegraphics[keepaspectratio, scale=0.3]{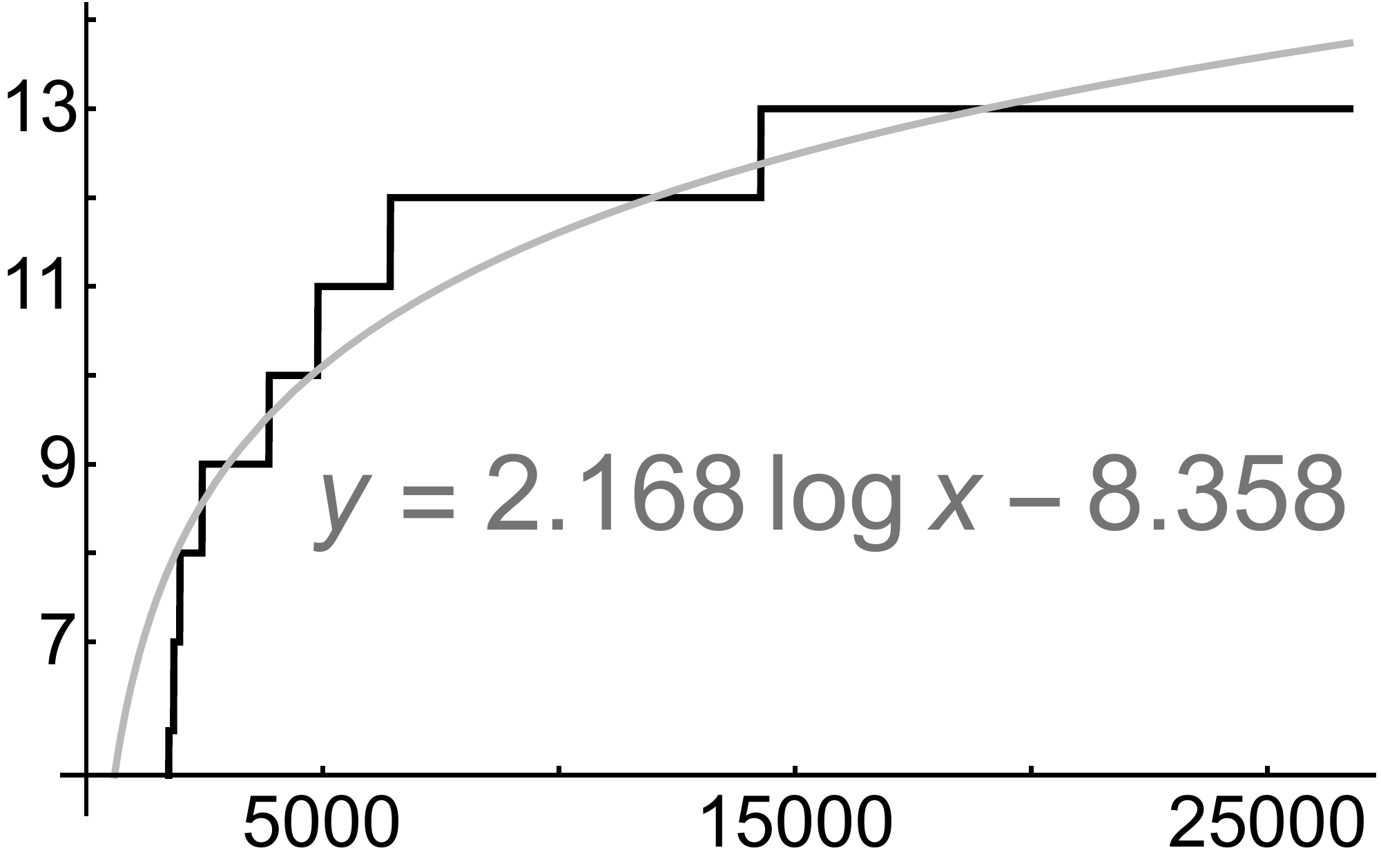}
    \caption*{$\alpha=0.3,\ x\leq 26750$}
  \end{minipage}
  \begin{minipage}[t]{0.48\linewidth}
    \centering
    \includegraphics[keepaspectratio, scale=0.3]{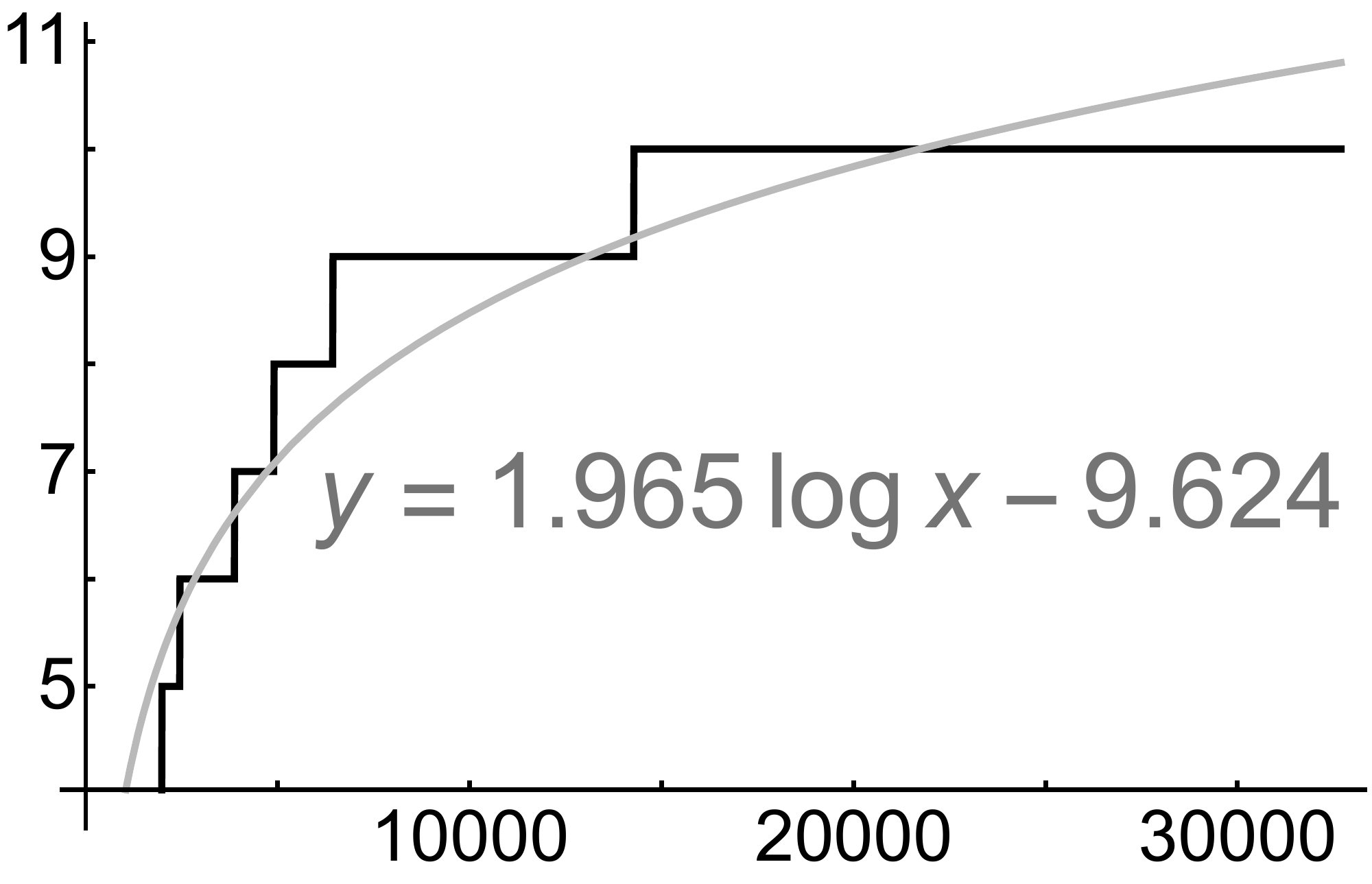}
    \caption*{$\alpha=0.4,\ x\leq 32700$}
  \end{minipage}
\end{figure}
For $\alpha\in \{0.5,0.55,\ldots, 1.0\}$, we compute $\# T_{\leq 10^4}(\alpha)$. As the graphs are similar to the cases of $\alpha\in \{0.3,0.4\}$, we have omitted them here and instead offer Table~\ref{Table1}. To describe the table, for $\alpha\in(0,1]$ and real $x>1$, we define
\begin{align*}
U_{\leq x}(\alpha)&=\Pas{(n_1,n_2,n_3)\in \mathbb{N}^3\colon (\floor{\alpha n_1^2},\floor{\alpha n_2^2},\floor{\alpha n_3^2})\in T_{\leq x}(\alpha)}.
\end{align*}
Note that $\# U_{\leq x}(\alpha) = \# T_{\leq x}(\alpha)$ for all $\alpha \in (0,1]$ and $x>1$.

\begin{table}[htbp]

\begin{center}
\caption{}\label{Table1}

\begin{tabular}{|c|c|l|}

\hline
 
   $\alpha$ &  $\# U_{\leq 10^4}(\alpha) $                                                                &   elements in $ U_{\leq 10^4}(\alpha)$                                                                  \\ \hline\hline
$0.50$     & $15$ & \scriptsize\begin{tabular}{l}
$(3,3,3)$, $(5,5,11)$, $(12,12,71)$, $(17,17,143)$, $(23,29,29)$, $(29,29,419)$,\\
$(70,70,168)$, $(70,70,2449)$, $(99,99,4899)$, $(461,1483,1741)$,\\
$(1219,1559,1847)$, $(1429,1597,1979)$, $(1886,2351,3871)$,\\
$(1887,2943,5247)$, $(2449,3433,6439)$
\end{tabular} \\ \hline
$0.55$     & $12$  & 
\scriptsize \begin{tabular}{l}
$(5,5,11)$, $(12,12,71)$, $(23,29,29)$, $(29,29,419)$, $(70,70,168)$, $(70,70,2449)$,\\ $(99,99,4899)$, $(461,1483,1741)$, $(1219,1559,1847)$, $(1429,1597,1979)$, \\
$(1886,2351,3871)$, $(2449,3433,6439)$
\end{tabular}  \\ \hline
$0.60$     & $9$  & 
\scriptsize
\begin{tabular}{l}
$(23,29,29)$, $(29,29,419)$, $(70,70,2449)$, $(99,99,4899)$, $(461,1483,1741)$,\\ 
$(1219,1559,1847)$, $(1429,1597,1979)$, $(1886,2351,3871)$, $(2449,3433,6439)$
\end{tabular}   \\ \hline
$0.65$     & $9$  & 
\scriptsize
\begin{tabular}{l}
$(23,29,29)$, $(29,29,419)$, $(70,70,2449)$, $(99,99,4899)$, $(461,1483,1741)$,\\ 
$(1219,1559,1847)$, $(1429,1597,1979)$, $(1886,2351,3871)$, $(2449,3433,6439)$
\end{tabular}  \\ \hline
$0.70$     & $9$ &
\scriptsize
\begin{tabular}{l}
$(23,29,29)$, $(29,29,419)$, $(70,70,2449)$, $(99,99,4899)$, $(461,1483,1741)$, \\
$(1219,1559,1847)$, $(1429,1597,1979)$, $(1886,2351,3871)$, $(2449,3433,6439)$

\end{tabular}  \\ \hline
$0.75$     & $15$   & 
\scriptsize
\begin{tabular}{l}
$(3,3,3)$, $(5,5,11)$, $(12,12,71)$, $(17,17,143)$ $(23,29,29)$, $(29,29,419)$,\\
$(70,70,168)$, $(70,70,2449)$, $(99,99,4899)$, $(461,1483,1741)$,\\
$(1219,1559,1847)$, $(1429,1597,1979)$, $(1886,2351,3871)$,\\
$(1887,2943,5247)$, $(2449,3433,6439)$
\end{tabular}
\\
\hline
$0.80$     & $9$ & 
\scriptsize
\begin{tabular}{l}
$(23,29,29)$, $(29,29,419)$, $(70,70,2449)$, $(99,99,4899)$, $(461,1483,1741)$, \\
$(1219,1559,1847)$, $(1429,1597,1979)$, $(1886,2351,3871)$, $(2449,3433,6439)$
\end{tabular}   \\ \hline
$0.85$     & $9$   & 
\scriptsize
\begin{tabular}{l}
$(23,29,29)$, $(29,29,419)$, $(70,70,2449)$, $(99,99,4899)$, $(461,1483,1741)$, \\
$(1219,1559,1847)$, $(1429,1597,1979)$, $(1886,2351,3871)$, $(2449,3433,6439)$

\end{tabular}    \\ \hline
$0.90$     & $9$  &
\scriptsize
\begin{tabular}{l}
$(23,29,29)$, $(29,29,419)$, $(70,70,2449)$, $(99,99,4899)$, $(461,1483,1741)$, \\
$(1219,1559,1847)$, $(1429,1597,1979)$, $(1886,2351,3871)$, $(2449,3433,6439)$

\end{tabular}  \\ \hline
$0.95$     & $9$ & 
\scriptsize
\begin{tabular}{l}
$(23,29,29)$, $(29,29,419)$, $(70,70,2449)$, $(99,99,4899)$, $(461,1483,1741)$, \\
$(1219,1559,1847)$, $(1429,1597,1979)$, $(1886,2351,3871)$, $(2449,3433,6439)$

\end{tabular}   \\ \hline
$1.00$     & $0$ & $\,$ none \\ \hline
\end{tabular}
\end{center}
\end{table}

To produce Figure~\ref{Figure1}, Figure~\ref{Figure2}, and Table~\ref{Table1}, we first applied \texttt{C++} to determine the lists of tuples in $U_{\leq x}(\alpha)$. We then applied \texttt{Mathematica} and checked that these tuples truly belong to $U_{\leq x}(\alpha)$ in order to guarantee the accuracy of the calculations. Figure~\ref{Figure1} and Figure~\ref{Figure2} were plotted by \texttt{Mathematica}; the approximated curves were also computed by \texttt{Mathematica}.

\begin{conjecture}\label{Conjecture-variantPEB}
There exist infinitely many  $(n_1,n_2,n_3)\in \N^3$ with $n_1,n_2,n_3$ $\geq 2$ such that all of $n_1^2+n_2^2-1$, $n_1^2+n_3^2-1$,  $n_2^2+n_3^2-1$, $n_1^2+n_2^2+n_3^2-2$ are perfect squares.
\end{conjecture}
It follows from Theorem~\ref{Theorem-lowerT-almostall} that for almost all $\alpha\in (0,1)$ there exist infinitely many tuples  $(n_1,n_2,n_3,m_1,m_2,m_3,m_4)\in \mathbb{N}^4$ such that
\begin{gather*}
    \floor{\alpha n_1^2} + \floor{\alpha n_2^2}=\floor{\alpha m_1^2}, \quad    \floor{\alpha n_2^2} + \floor{\alpha n_3^2}=\floor{\alpha m_2^2}, \\
    \floor{\alpha n_3^2} + \floor{\alpha n_1^2}=\floor{\alpha m_3^2}, \quad \floor{\alpha n_1^2} + \floor{\alpha n_2^2}+\floor{\alpha n_3^2}=\floor{\alpha m_4^2}.
\end{gather*}
Since $\lim_{\alpha\to1^-}\floor{\alpha n^2}=n^2-1$, we anticipate that there exist infinitely many tuples $(n_1,n_2,n_3,m_1,m_2,m_3,m_4)\in \mathbb{N}^4$ such that
\begin{gather*}
    n_1^2 +  n_2^2-1=m_1^2, \quad    n_2^2 + n_3^2-1=m_2^2, \\
    n_3^2 +  n_1^2-1=m_3^2, \quad n_1^2 + n_2^2+n_3^2-2=m_4^2.
\end{gather*}
Therefore, Conjecture~\ref{Conjecture-variantPEB} is expected from Theorem~\ref{Theorem-lowerT-almostall}. Let $V(x)$ be the set of all tuples $(n_1,n_2,n_3)\in \mathbb{N}^3$ with $n_1\leq n_2\leq n_3\leq x$ such that all of
\[
n_1^2+n_2^2-1,\quad n_1^2+n_3^2-1,\quad  n_2^2+n_3^2-1,\quad n_1^2+n_2^2+n_3^2-2
\]
are squares. Although it is believed that there is no perfect Euler brick, this conjecture suggests that there exists an infinite number of variants of perfect Euler bricks. By numerical calculation, we have 
\begin{align*}
&V(46300)\\
=&\{(23,29,29),(461,1483,1741),(1219,1559,1847),(1429,1597,1979), \\
&(1886,2351,3871),(2449,3433,6439),(3095,18880,20351) \}.
\end{align*}
We do not know the reason, but it can be seen that  $V(46300)\subset U_{46300}(0.5)$.

\appendix
\section{Ceiling Functions}
In this appendix, we examine the set 
\[
\upperS(\alpha) :=\{\lceil \alpha n^2 \rceil \colon n\in \mathbb{N}\}
\]
for every $\alpha\in \mathbb{R}$. Let $\upperT(\alpha)=T(\upperS(\alpha))$. Note that $\upperS(1)=S$ and that for every fixed $n\in \mathbb{N}$, $\lim_{\alpha\to 1^-} \lceil  \alpha n^2  \rceil= n^2$. Since $\upperS(\alpha)$ is a discrete set, there is no gap between $\lceil  \alpha n^2  \rceil$ and $n^2$ if $n$ is sufficiently large. We expect that we would  obtain Theorem~\ref{Theorem-lowerT-Q} and  Theorem~\ref{Theorem-lowerT-almostall} for $\upperS(\alpha)$ in a similar manner to the proofs of these theorems. However, the problems involving $\upperS(\alpha)$ are much harder. The small difference between $S(\alpha)$ and $\upperS(\alpha)$ causes strange phenomena.
\begin{theorem}\label{Theorem-upperT-Q}
We have $\#\upperT(p/P_q)=\infty$ for all odd $q\in\N$ and $p\in [0,(4/9)P_q)_{\Z}$.
\end{theorem}
\begin{theorem}\label{Theorem-upperT-QQ}
We have $\#\upperT(\alpha)=\infty$ for all rational numbers 
\[
\alpha\in (1/8,3/16]\cup(1/3,3/8]\cup(5/9,9/16].
\]
\end{theorem}

Moreover, we obtain the following theorem as a metric result.

\begin{theorem}\label{Theorem-upper-Lebesguemain1}
 We have $\#\overline{T}(\alpha)=\infty$ for almost all $\alpha \in (0,2/3)$. 
\end{theorem}
As a consequence, we get the following corollary.

\begin{corollary}\label{Corollary-upperT-positive}
There exists an open and dense set $U\subseteq (0,2/3)$ such that $\#\overline{T}(\alpha)>0$ for all $\alpha\in U$.
\end{corollary}
The proof of Corollary~\ref{Corollary-upperT-positive} is similar to that of Corollary~\ref{Corollary-lowerT-positive}. For all $\alpha\in (0,2/3)$ with $\#\upperT(\alpha)=\infty$, there exists $\epsilon(\alpha)>0$ such that any $\gamma \in (\alpha-\epsilon(\alpha),\alpha)$ satisfies  $\#\upperT(\gamma)>0$. Let
\[
U=\bigcup_{\substack {\alpha\in (0,2/3) \\ \#\upperT(\alpha)=\infty}  } (\alpha-\epsilon(\alpha),\alpha)   \cap (0,2/3).
\]
By Theorem~\ref{Theorem-upper-Lebesguemain1}, $U$ is open and dense in $(0,2/3)$. Therefore, we deduce Corollary~\ref{Corollary-upperT-positive}.

\begin{proof}[Proof of Theorem~\ref{Theorem-upperT-Q}]
Let $q$ be an arbitrary odd positive integer and $p\in [1,(4/9)P_q)_\Z$. Let us show that for every odd $n\in \mathbb{N}$,
		\begin{equation}\label{Relation-upperT}
		\pas{\ceil{\frac{p}{P_q}P_{qn-1}^2}, \ceil{\frac{p}{P_q}P_{qn-1}^2}, \ceil{\frac{p}{P_q}\pas{\frac{P_{qn-1}^2}{2}-1}^2}}\in T\pas{\frac{p}{P_q}}.
		\end{equation}
Take any odd $n\in \mathbb{N}$. We define $x=P_{qn}-P_{qn-1}$ and $y=P_{qn-1}$. Since $qn-1$ is even,  Lemma~\ref{Lemma-PellEquation} yields $x^2=2y^2+1$. It follows from \eqref{Property-Pell2} in Lemma~\ref{Lemma-Pell} that  
\[
y^2=P_{qn-1}^2=P_{qn}P_{qn-2}+ (-1)^{qn}=P_{qn}P_{qn-2}-1.
\]
By \eqref{Property-Pell1} in Lemma~\eqref{Lemma-Pell}, $P_q \mid P_{qn}$. Therefore, we have
	\begin{align*}
		\ceil{\frac{p}{P_q} x^2}&=\ceil{\frac{p}{P_q}(2y^2+1)}=\ceil{2p\frac{P_{qn}}{P_q} P_{qn-2}-\frac{p}{P_q}}\\
		&=2\frac{p}{P_q}P_{qn}P_{qn-2}=2\frac{p}{P_q}\pas{y^2+1}.
	\end{align*}
	Note that $(p/P_q) (y^2+1)$ is an integer which is equal to $p (P_{qn}/P_q)P_{qn-2}$ and we have
	\[
	\frac{p}{P_q}(y^2+1)-1<\frac{p}{P_q} y^2\leq \frac{p}{P_q} (y^2+1).
	\]
	Therefore, we obtain 
	\begin{equation}\label{Equation-upperT-Q1}
	\ceil{\frac{p}{P_q} x^2}	=2\ceil{\frac{p}{P_q} y^2}.
\end{equation}	
	In addition, since $P_q\mid y^2+1$ and $\ceil{(p/P_q)y^2}= \ceil{(p/P_q)(y^2+1)}$, we get
	\begin{equation}\label{Equation-upperT-Q2}
		\ceil{\frac{p}{P_q} \pas{\frac{y^2}{2}+1}^2}=\ceil{\frac{p}{P_q}\frac{y^4}{4}}+\ceil{\frac{p}{P_q}(y^2+1)}=\ceil{\frac{p}{P_q} \pas{\frac{y^2}{2}}^2}+\ceil{\frac{p}{P_q} y^2}.
	\end{equation}
	Let us now show that
	\begin{equation}\label{Equation-upperT-Q3}
	    \ceil{\frac{p}{P_q} \pas{y^2/2-1}^2}  =\ceil{\frac{p}{P_q} \pas{y^2/2}^2}-\ceil{\frac{p}{P_q} y^2}.
	\end{equation}
By $y^2=P_{qn}P_{qn-2}-1$ and $p\in [1,(4/9)P_q)_{\mathbb{Z}}$, we have
	\begin{gather*}
	\Pas{\frac{p}{P_q}y^2}=1-\frac{p}{P_q},\\
	\Pas{\frac{p}{P_q}\pas{\frac{y^2}{2}-1}^2}=\Pas{\frac{p}{P_q} \frac{y^4}{4}-\frac{p}{P_q}y^2+\frac{p}{P_q} }
	=  \Pas{\frac{9p}{4P_q}}
	=\frac{9p}{4P_q} ,
	\end{gather*}
which implies that $1<\Pas{(p/P_q) y^2}+\Pas{(p/P_q) \pas{y^2/2-1}^2}<2$.
Hence, we have
	\begin{align*}
	\ceil{\frac{p}{P_q}\pas{y^4/4+1}}&= 
	\ceil{\frac{p}{P_q}\pas{y^2/2-1}^2}+ \ceil{\frac{p}{P_q} y^2}.
	\end{align*}
	Additionally, by $p<(4/5)P_q$, the left-hand side is equal to $\ceil{(p/P_q) (y^2/2)^2}$. Therefore, we get \eqref{Equation-upperT-Q3}. Similar to the proof of Theorem~\ref{Theorem-lowerT-Q}, we can deduce \eqref{Relation-upperT} from \eqref{Equation-upperT-Q1}, \eqref{Equation-upperT-Q2}, and  \eqref{Equation-upperT-Q3}. 
	\end{proof}
\begin{proof}[Proof of Theorem~\ref{Theorem-upperT-QQ}]
Let $\alpha$ be an arbitrary rational number in $(1/8,3/16]\cup(1/3,3/8]\cup(5/9,9/16]$. Let $\alpha=p/q$ where $p,q\in \mathbb{N}$ are relatively prime. Take $r=r(q)\in \mathbb{N}$ as in Lemma~\ref{Lemma-divisibility}. 
We can show that for all $n\in \mathbb{N}$ with $4\mid n$,
	\begin{equation}\label{Relation-upperT-QQ}
	\pas{\ceil{\alpha P_{rn-2}^2}, \ceil{\alpha P_{rn-2}^2},
		\ceil{\alpha\pas{\frac{P_{rn-2}^2}{2}}^2}}\in T\pas{\alpha}.
\end{equation}
 Fix any $n\in \mathbb{N}$ with $4\mid n$. We define $x=P_{rn-1}-P_{rn-2}$ and $y=P_{rn-2}$. Since $rn-2$ is even, we get $x^2=2y^2+1$ and 
\[
y^2=(P_{rn}-2P_{rn-1})^2\equiv 4P_{rn-1}^2=4(P_{rn}P_{rn-2}+1)\equiv4\mod 4q.
\]
For all $x\in \mathbb{R}$, we define $\delta(x)=\{x\}$ if $x\notin \mathbb{Z}$, and  $\delta(x)=1$ otherwise. Then we obtain $\lceil x\rceil =x+1-\delta(x)$. Therefore, the following equivalences hold: 
\begin{align*}
	&\ceil{\alpha x^2}=2\ceil{\alpha y^2}\\
	&\iff \alpha(2y^2+1)+1-\delta\pas{\alpha(2y^2+1)}=2\alpha y^2+2-2\delta\pas{\alpha y^2}\\
	&\iff 1-\alpha+\delta\pas{9\alpha}-2\delta\pas{4\alpha}=0,\\
	&\ceil{\alpha\pas{y^2/2+1}^2}=\ceil{\alpha\pas{y^2/2}^2}+\ceil{\alpha y^2}\\
	&\iff \alpha\pas{y^2/2+1}^2+1-\delta\pas{\alpha\pas{y^2/2+1}^2}\\
	&\hspace{30pt} =\alpha\pas{y^2/2}^2+\alpha y^2+2-\delta\pas{\alpha\pas{y^2/2}^2}-\delta\pas{\alpha y^2}\\
	&\iff 1-\alpha+\delta(9\alpha)-2\delta(4\alpha)=0,\\
	&\ceil{\alpha\pas{y^2/2+2}^2}=\ceil{\alpha\pas{y^2/2+1}^2}+\ceil{\alpha y^2}\\
	&\iff \alpha\pas{y^2/2+2}^2+1-\delta\pas{\alpha\pas{y^2/2+2}^2}\\
	&\hspace{30pt}=\alpha\pas{y^2/2+1}^2+\alpha y^2+2-\delta\pas{\alpha\pas{y^2/2+1}^2}-\delta\pas{\alpha y^2}\\
	&\iff 1-3\alpha+\delta\pas{16\alpha}-\delta\pas{9\alpha}-\delta\pas{4\alpha}=0.
\end{align*}
Since $\alpha\in (1/8,3/16]\cup(1/3,3/8]\cup(5/9,9/16]$, we   have $1-\alpha+\delta\pas{9\alpha}-2\delta\pas{4\alpha}=0$ and $1-3\alpha+\delta\pas{16\alpha}-\delta\pas{9\alpha}-\delta\pas{4\alpha}=0$. Hence, we conclude \eqref{Relation-upperT-QQ}. 
\end{proof}

Before proving Theorem~\ref{Theorem-upper-Lebesguemain1}, we present the following lemma. 
\begin{lemma}\label{Lemma-range}Let $0<s<t<1$. We define
\begin{align*}
B_1&=\{(u,v)\in \mathbb{R}^2\colon (1-s)/2< u< (2-t)/2, v>0 \}\\
B_2&=\{(u,v)\in \mathbb{R}^2\colon 1-s< u+v< 2-t \} ,\\
B_3&=\{(u,v)\in \mathbb{R}^2\colon t<u-v< 1+s \}. 
\end{align*}
Let $\alpha \in [s,t]$. If $n\in \mathbb{N}$ satisfies 
\begin{equation}\label{Relation-B's}
( \{\alpha P_{2n}^2\},\{\alpha P_{2n}^4/4\}) \in B_1\cap B_2\cap B_3\cap (0,1)^2,
\end{equation}
then we have $(\ceil{ \alpha P_{2n}^2}, \ceil{ \alpha P_{2n}^2}, \ceil{ \alpha (P_{2n}^2/2-1)^2 } ) \in \upperT (\alpha)$. 
\end{lemma}

\begin{remark}Each $B_i$ is non-empty. Indeed, it follows that
\begin{align*}
(2-t)/2-(1-s)/2 &= 1- (t-s)/2>0, \\
(2-t)-(1-s) &=1-(t-s)>0,\\
(1+s)-t&=1-(t-s)>0.
\end{align*}
It is non-trivial that $B_1\cap B_2\cap B_3\cap (0,1)^2\neq \emptyset$. We will verify this in the next section.   
\end{remark}

\begin{proof}[Proof of Lemma~\ref{Lemma-range}]
Fix any $n\in \mathbb{N}$ satisfying \eqref{Relation-B's}. 
Let $x=x_n=P_{2n+1}-P_{2n}$ and $y=y_{n}=P_{2n}$. Then $x^2=2y^2+1$ holds. Since $\alpha y^2 $ is not an integer by the definition of $B_1$, we have
\[
\lceil \alpha y^2 \rceil = \alpha y^2  +1 -\{\alpha y^2 \}.
\]
Therefore, we observe that 
\begin{align}\nonumber
    \lceil \alpha x^2 \rceil = 2\lceil \alpha y^2 \rceil &\iff 2\lceil \alpha y^2 \rceil-1 < 2\alpha y^2+\alpha  \leq 2\lceil \alpha y^2 \rceil\\ \label{Inequality-Apendix1}
    &\iff 1-2\{\alpha y^2\}<\alpha \leq 2-2\{\alpha y^2\}.
\end{align}
By the definition of $B_1$ and the choice of $n$, we have 
\[
(1-s)/2<\{\alpha y^2 \}< (2-t)/2,
\]
which implies \eqref{Inequality-Apendix1} since $\alpha\in [s,t]$. Therefore, we deduce $\lceil \alpha x^2 \rceil = 2\lceil \alpha y^2 \rceil$. Let us next show that 
\begin{equation}\label{Equation-AP1}
\lceil \alpha (y^2/2+1)^2 \rceil = \lceil \alpha (y^2/2)^2 \rceil +\lceil \alpha y^2 \rceil.
\end{equation}
Since $\alpha (y^2/2)^2 $ is not an integer from the definition of $B_1$, we obtain
\[
\lceil\alpha (y^2/2)^2  \rceil = \alpha (y^2/2)^2   +1 -\{\alpha (y^2/2)^2  \}.
\]
Therefore, we also observe that 
\begin{align*}
    \eqref{Equation-AP1} 
    &\iff  \lceil \alpha (y^2/2)^2 \rceil +\lceil \alpha y^2 \rceil-1 < 
    \alpha \frac{y^4}{4} +\alpha (y^2+1) \leq  \lceil \alpha (y^2/2)^2 \rceil +\lceil \alpha y^2 \rceil\\
    &\iff 1-(\{ \alpha y^2 \}+\{ \alpha (y^2/2)^2 \} ) <\alpha \leq 2-(\{ \alpha y^2 \}+\{ \alpha (y^2/2)^2 \} )
\end{align*}
By the definition of $B_2$ and the choice of $n$, we have
\[
1-s< \{ \alpha y^2 \}+\{ \alpha (y^2/2)^2 \} < 2-t,
\]
which implies the above equivalent condition to \eqref{Equation-AP1} since $\alpha\in [s,t] $. Thus, we conclude \eqref{Equation-AP1}. We lastly show that 
\begin{equation}\label{Equation-AP2}
\lceil \alpha (y^2/2-1)^2 \rceil = \lceil \alpha (y^2/2)^2 \rceil -\lceil \alpha y^2 \rceil.
\end{equation}
Similarly, we obtain the following equivalences:
\begin{align*}
    \eqref{Equation-AP2} 
    &\iff  \lceil \alpha (y^2/2)^2 \rceil -\lceil \alpha y^2 \rceil-1 < 
    \alpha \frac{y^4}{4} -\alpha y^2+\alpha \leq  \lceil \alpha (y^2/2)^2 \rceil -\lceil \alpha y^2 \rceil\\
    &\iff -1+\{ \alpha y^2 \}-\{ \alpha (y^2/2)^2 \}  <\alpha \leq \{ \alpha y^2 \}-\{ \alpha (y^2/2)^2 \} 
\end{align*}
By the definition of $B_3$ and the choice of $n$, we obtain
\[
 t< \{ \alpha y^2 \}- \{ \alpha (y^2/2)^2 \}< 1+s,
\]
which implies the above equivalent condition to \eqref{Equation-AP2} since $\alpha\in[s,t]$. Therefore, we conclude \eqref{Equation-AP2}.

By the above discussion, setting $k=\ell = \lceil \alpha y^2 \rceil$ and $ m=\lceil \alpha (y^2/2-1)^2 \rceil$, we obtain $(k,\ell, m)\in \upperT(\alpha)$.  
\end{proof}

\begin{proof}[Proof of Theorem~\ref{Theorem-upper-Lebesguemain1}]
Fix any $0<t<2/3$, and let $0<s<t$ be a parameter chosen later. Let $B_1$, $B_2$, $B_3$ be as in Lemma~\ref{Lemma-range}. Lemma~\ref{Lemma-uniformly} shows that for almost all $\alpha\in [s,t]$, 
\begin{align*}
 &\lim_{N\to \infty}  \frac{\#\left\{ n\in [N] \colon 
 ( \{\alpha P_{2n}^2\},\{\alpha P_{2n}^4/4\}) \in B_1\cap B_2\cap B_3\cap (0,1)^2
 \right\}}{N}\\ 
 &= \mu_2 (B_1\cap B_2\cap B_3\cap (0,1)^2). 
\end{align*}
Therefore, by Lemma~\ref{Lemma-range}, if $\mu (B_1\cap B_2\cap B_3\cap (0,1)^2)>0$, then $\#\upperT (\alpha)=\infty$ for almost all $\alpha \in [s,t]$. Let us show that $\mu (B_1\cap B_2\cap B_3\cap (0,1)^2)>0$. Note that $B_1$, $B_2$, $B_3$, and $(0,1)^2$ are open. Therefore, it suffices to show that $B_1\cap B_2\cap B_3\cap (0,1)^2$ is non-empty. Let $R= B_2\cap B_3$.  By solving the systems of linear equations, $R$ forms  the open square whose vertices are 
\begin{align*}
\left(\frac{1+(t-s)}{2},\ \frac{1-(t+s)}{2}\right)&,\  \left(1,\ 1-t\right),\\
\left(\frac{3-(s+t)}{2},\ \frac{1-(t+s)}{2}\right)&,\ \left(1,\ -s\right).
\end{align*}
Define $\mathrm{proj}_1\colon \mathbb{R}^2 \to \mathbb{R}$ as $\mathrm{proj}_1(u,v)=u$. Let us consider two cases as follows:
(a) $0<t\leq 1/2$, (b) $t>1/2$.
In case (a), we choose an arbitrarily small $s\in (0,t)$. Then we obtain $(1-(t+s))/2\in (0,1)$. Therefore,
\begin{align*}
&\mathrm{proj}_1(R\cap(0,1)^2\cap B_1)\\
&=  \left(\frac{1+(t-s)}{2},\ \frac{3-(s+t)}{2} \right) \cap (0,1)\cap \left(\frac{1-s}{2}, \frac{2-t}{2}  \right)\\
& = \left(\frac{1+(t-s)}{2},\ 1 \right) \cap \left(\frac{1-s}{2}, \frac{2-t}{2}  \right)
\end{align*}
which is non-empty by $0<s<t\leq 1/2$. Thus, almost all $\alpha\in (0,1/2]$ satisfy $\# \upperT(\alpha)=\infty$. In case (b), we choose any $s\in (1/2,t)$. Then  $(1-(t+s))/2<0$. Therefore,
\begin{align*}
\mathrm{proj}_1(R\cap(0,1)^2\cap B_1)&=  \left(t, 1 \right) \cap \left(\frac{1-s}{2}, \frac{2-t}{2}  \right),
\end{align*}
which is non-empty from $1/2<t<2/3$. Hence, we can conclude that $\# \upperT(\alpha)=\infty$ for almost all $\alpha \in [1/2,2/3]$.   
\end{proof}

\begin{remark}
For every $\alpha>0$ and $x>1$, we define $\upperT_{\leq x}(\alpha)$ as the number of tuples $(\ceil{\alpha n_1^2} ,\ceil{\alpha n_2^2},\ceil{\alpha n_3^2})\in \upperT(\alpha)$ with $n_1\leq n_2\leq n_3\leq x$. From~\eqref{Relation-upperT}, we deduce that 
\[
\# \upperT_{\leq x}(p/P_q) \geq  \floor{ \frac{\log(16\phi_P^2 x) }{2q\log \phi_P }} 
\]
for all odd $q\in \mathbb{N}$, $p\in [1,(4/9)P_q)_{\mathbb{Z}}$, and $x>1$. We can also get similar lower bounds in the case of $\alpha\in ((1/8,3/16]\cup(1/3,3/8]\cup(5/9,9/16])\cap \mathbb{Q}$. In addition, let $0<s<t<1$;  lemma~\ref{Lemma-range} and the proof of Theorem~\ref{Theorem-upper-Lebesguemain1} imply that for almost all $\alpha\in [s,t]$ there exist $\lambda'=\lambda'(s,t)>0$ and $x_0=x_0(s,t,\alpha)>0$ such that for all $x\geq x_0$, we obtain
$\# \upperT_{\leq x}(\alpha) \geq \lambda' \log x$.
\end{remark}

\section*{Acknowledgement} We would like to thank the referee for valuable comments. The second author was supported by JSPS KAKENHI Grant Number JP22J00025.

\bibliographystyle{amsalpha}
\bibliography{references_PEB}
\end{document}